\documentclass[12pt,reqno]{amsart}

\usepackage[T1]{fontenc}
\usepackage{amsmath}
\usepackage{amssymb}
\usepackage{amsfonts}
\usepackage{graphicx}
\usepackage{color}
\usepackage[bookmarksnumbered, colorlinks, plainpages]{hyperref}
\hypersetup{colorlinks=true,linkcolor=blue, anchorcolor=green, citecolor=violet, urlcolor=cyan, filecolor=magenta, pdftoolbar=true}
\usepackage{standalone}
\usepackage{enumitem}
\usepackage{tikz-cd}
\usepackage{mdframed}
\usepackage{stmaryrd}

\makeatletter \@mparswitchfalse \makeatother

\usepackage[left=1.1in, right=1.1in, bottom=1.5in]{geometry}
\newtheorem{theo}{Theorem}[section]
\newtheorem{lemm}[theo]{Lemma}
\newtheorem{fact}{Fact}
\newtheorem{prop}[theo]{Proposition}
\newtheorem{coro}[theo]{Corollary}
\theoremstyle{definition}
\newtheorem{defi}[theo]{Definition}

\theoremstyle{remark}
\newtheorem{rem}[theo]{Remark}
\numberwithin{equation}{section}

\def\I{{\mathbb I}}

\def\N{{\mathbb N}}

\def\R{{\mathbb R}}

\def\DD{{\mathbb D}}
\def\EE{{\mathbb E}}
\def\FF{{\mathbb F}}
\def\II{{\mathbb I}}

\def\PP{{\mathbb P}}
\def\QQ{{\mathbb Q}}
\def\TT{{\mathbb T}}

\def\pp{\leq}
\def\pg{\geq}

\DeclareMathAlphabet\cat{OMS}{cmsy}{b}{n}

\begin{document}

\title[$K$-closedness results in noncommutative Lebesgue spaces with filtrations]{$K$-closedness results in noncommutative Lebesgue spaces with filtrations}

\author[Moyart Hugues]{Hugues Moyart}
\address{UNICAEN, CNRS, LMNO, 14000 Caen, France}
\email{hugues.moyart@unicaen.fr}
 


\renewcommand{\baselinestretch}{1.2}\normalsize

\begin{abstract}
In this paper, we establish a new general
$K$-closedness result in the context of real interpolation of noncommutative Lebesgue spaces involving filtrations. As an application, we derive $K$-closedness results for various classes of noncommutative martingale Hardy spaces, addressing a problem raised by Randrianantoanina. The proof of this general result adapts Bourgain's approach to the real interpolation of classical Hardy spaces on the disk within the framework of noncommutative martingales.
\end{abstract} 

\maketitle

\renewcommand{\baselinestretch}{0.1}\normalsize

\tableofcontents

\renewcommand{\baselinestretch}{1.2}\normalsize

\section*{Introduction}

The paper is motivated by advances in the context of real interpolation of classical Hardy spaces on the disk following the work of Peter Jones. Let us review the results obtained in this context. Let $\TT$ be the unit circle, and let $H_p(\TT)\subset L_p(\TT)$ denote, for $1\pp p\pp\infty$, the class of classical Hardy spaces on the unit disk. Peter Jones established in \cite{JonesHardy2} that there is a universal constant $C>0$ such that for every $f\in H_1(\TT)+H_\infty(\TT)$, and $t>0$, we have $K(t,f,H_1(\TT),H_\infty(\TT))\pp CK(t,f,L_1(\TT),L_\infty(\TT))$, where $K$ refers to Peetre's $K$-functional in the context of real interpolation theory. According to the terminology introduced by Pisier in \cite{PisierHardy}, one can reformulate Jones' theorem by saying that the Banach couple $(H_1(\TT),H_\infty(\TT))$ is $K$-closed in the Banach couple $(L_1(\TT),L_\infty(\TT))$. This includes the fact that for every $0<\theta<1$, we have
\[(H_1(\TT),H_\infty(\TT))_{\theta,p}=H_p(\TT)\]
with equivalent norms, where $1/p=1-\theta$, and where the notation on the left-hand side refers to the real interpolation method. In contrast with the existing proofs of Jones' result in the extensive literature devoted to the real interpolation of classical Hardy spaces on the disk so far, Bourgain was able to substitute complex variable techniques by real variable methods. The approach of Bourgain to Jones' theorem is essentially based on the fact that the Riesz projection $L_2(\TT)\to H_2(\TT)$, (i.e., the orthogonal projection of $L_2(\TT)$ onto $H_2(\TT)$) is a Calderón-Zygmund singular integral operator. Using the Calderón-Zygmund decomposition, he established that the Banach couple $(H_1(\TT),H_p(\TT))$ is $K$-closed in the Banach couple $(L_1(\TT),L_p(\TT))$, for every $1<p<\infty$. Then, using an abstract duality lemma for $K$-closedness due to Pisier in \cite{PisierHardy}, he deduced that the Banach couple $(H_p(\TT),H_\infty(\TT))$ is $K$-closed in the Banach couple $(L_p(\TT),L_\infty(\TT))$, for every $1<p<\infty$. As highlighted by Kislyakov and Xu in \cite{Kislyakov}, who established an abstract Wolff-type interpolation result for $K$-closedness, one can deduce Jones' theorem from the two partial results of Bourgain above. 

In this paper, we adapt Bourgain's approach within the framework of noncommutative martingales. The key ingredient is Gundy’s decomposition theorem for noncommutative martingales proved by Parcet and Randrianantoanina in \cite{ParcetGundy}. Indeed, it can be regarded as a noncommutative counterpart of the Calderón-Zygmund decomposition of functions. To better explain our considerations, we now introduce the mathematical setting of the paper. We refer to the body of the paper for unexplained notations in the following. Let $\mathcal{M}$ be a von Neumann algebra equipped with a normal faithful semifinite trace. By extracting the arguments that allowed Bourgain to prove his fundamental lemma \cite{Bourgain}[Lemma 2.4], we introduce a class of projections $\PP:D\to D$ defined on a subspace $D\subset L_1(\mathcal{M})\cap L_2(\mathcal{M})$ which is dense in both $L_1(\mathcal{M})$ and $L_2(\mathcal{M})$, such that, if $\PP_1(\mathcal{M})$ and $\PP_2(\mathcal{M})$, denote the norm closure of $\PP(D)$ in $L_1(\mathcal{M})$ and $L_2(\mathcal{M})$, respectively, then the Banach couple $(\PP_1(\mathcal{M}),\PP_2(\mathcal{M}))$ is $K$-closed in the Banach couple $(L_1(\mathcal{M}),L_2(\mathcal{M}))$. The first main contribution of the paper consists of the proof that this class of projections contains many examples that naturally arise when $\mathcal{M}$ is equipped with one or several filtrations. The decisive step in our arguments is the use of Gundy's decomposition for noncommutative martingales. By using Pisier's duality lemma for $K$-closedness and the Wolff-type result for $K$-closedness, we finally obtain new $K$-closedness results in noncommutative Lebesgue spaces that can be applied in various situations involving filtrations. In particular, we obtain a $K$-closedness result which can be considered as an extension to the range $[1,\infty]$ of the boundedness properties of noncommutative martingale transforms as established by Randrianantoanina in \cite{NarcisseMartingaleTransforms}.

The paper details some applications in relation to noncommutative martingale Hardy spaces. We use the framework of row/column/mixed spaces defined by Pisier and Xu in \cite{PisierXuMartingales} and denoted $L_p(\mathcal{M},\ell_2)$ for $1\pp p\pp\infty$ in this paper. If $L_p^{\rm ad}(\mathcal{M},\ell_2)\subset L_p(\mathcal{M},\ell_2)$ denotes the norm-closed subspace of adapted sequences, we establish that the Banach couple $(L_p^{\rm ad}(\mathcal{M},\ell_2),L_q^{\rm ad}(\mathcal{M},\ell_2))$ is $K$-closed in the Banach couple $(L_p(\mathcal{M},\ell_2),L_q(\mathcal{M},\ell_2))$, for every $1\pp p,q\pp\infty$. This result was first proved by Randrianantoanina in \cite{Narcisse2023}, and can be interpreted as an extension of the noncommutative version of Stein's inequality in the context of noncommutative martingale inequalities to the range $[1,\infty]$ and was originally proved by Pisier and Xu in \cite{PisierXuMartingales}. In the setting of row or column spaces, this result is contained in \cite{Narcisse2023}, and even extended to the full range $(0,\infty]$, but the approach allows us to encompass the case of mixed spaces as well. Actually, the approach of this paper extends the alternative proof by Xu contained in \cite{XuBook} of the noncommutative version of Stein's inequality, and that relies on the boundedness properties of noncommutative martingale transforms. It seemed all the more important to us to include this part because we will find all the reasoning and calculations carried out there for the last main result of the paper. Namely, if $H_p(\mathcal{M})\subset L_p(\mathcal{M},\ell_2)$ denotes the norm-closed subspace of martingale difference sequences, with the same strategy we establish that the Banach couple $(H_p(\mathcal{M}),H_q(\mathcal{M}))$ is $K$-closed in the Banach couple $(L_p(\mathcal{M},\ell_2),L_q(\mathcal{M},\ell_2))$, for every $1\pp p,q\pp\infty$. This answers a problem raised by  Randrianantoanina in \cite{Narcisse2023}. As a by-product of our result, we find that for every $0<\theta<1$, we have
\[(H_1(\mathcal{M}),H_\infty(\mathcal{M}))_{\theta,p}=H_p(\mathcal{M})\]
with equivalent norms, where $1/p=1-\theta$. In the setting of column spaces or row spaces, this result has also recently been proved by Randrianantoanina in \cite{Narcisse2024} by taking a different approach. We refer to the last section of the paper for the relationships between the spaces $H_p(\mathcal{M})$ and the usual noncommutative martingale Hardy spaces in relation to the noncommutative Burkholder-Gundy inequality.

The paper is organised as follows. In Section 1, we gather the mathematical background that will be needed in the sequel. In Section 2 we adapt Bourgain's approach to the real interpolation of classical Hardy spaces on the disk within the framework of noncommutative martingales, which allows us to state and prove general $K$-closedness results in noncommutative Lebesgue spaces. In Section 3 we use the previous material to derive new $K$-closedness results for noncommutative adapted spaces and noncommutative martingale Hardy spaces.

\section{Preliminaries}

For the sake of completeness and for the convenience of the reader, we gather in this section the mathematical background that will be needed in the sequel. 

The section is organised as follows. The first paragraph is devoted to the machinery of real interpolation theory. In the second paragraph, we gather the properties of noncommutative Lebesgue spaces, and more generally noncommutative fully symmetric spaces, in relation to real interpolation methods. In the third paragraph, we recall the basics of noncommutative martingale theory.

\subsection{Interpolation theory} 

We refer to \cite{BerghInterpolation} for a background on interpolation theory.

\subsubsection{General interpolation spaces}

Let $(E_0,E_1)$ be a \textit{Banach couple}, which means that $E_0,E_1$ are Banach spaces together with continuous and linear inclusions in a common Hausdorff topological vector space. Then the intersection space $E_0\cap E_1$ and the sum space $E_0+E_1$ are well defined Banach spaces according to the intersection norm $\|\cdot\|_{E_0\cap E_1}$ and the sum norm $\|\cdot\|_{E_0+E_1}$ defined as follows, 
\[\|u\|_{E_0\cap E_1}:=\max\big\{\|u\|_{E_0},\ \|u\|_{E_1}\big\},\]
\[\|u\|_{E_0+E_1}:=\inf\big\{\|u_0\|_{E_0}+\|u_1\|_{E_1}\ \mid\  u=u_0+u_1,\ u_0\in E_0,u_1\in E_1\big\}.\]
An \textit{intermediate space} for $(E_0,E_1)$ is a Banach space $E_\theta$ together with continuous and linear inclusions $E_0\cap E_1\subset E_\theta\subset E_0+E_1$.

Let $(E_0,E_1)$, $(F_0,F_1)$ be Banach couples. By a \textit{bounded operator} $T:(E_0,E_1)\to(F_0,F_1)$ we mean an operator $T:E_0+E_1\to F_0+F_1$ such that $T(E_0)\subset F_0,T(E_1)\subset F_1$, and such that the induced operators $T:E_0\to F_0,T:E_1\to F_1$ are bounded, and in that case we set 
\[\|T\|_{(E_0,E_1)\to(F_0,F_1)}:=\max\big\{\|T\|_{E_0\to F_0},\|T\|_{E_1\to F_1}\big\}.\]
An \textit{interpolation pair} for a pair of Banach couples $(E_0,E_1)$, $(F_0,F_1)$ is a pair of intermediate spaces $(E_\theta,F_\theta)$ for $(E_0,E_1)$ and $(F_0,F_1)$ respectively, such that if $T:(E_0,E_1)\to(F_0,F_1)$ is a bounded operator, then $T(E_\theta)\subset F_\theta$ and the induced operator $T:E_\theta\to F_\theta$ is bounded, with 
\[\|T\|_{E_\theta\to F_\theta}\pp C\|T\|_{(E_0,E_1)\to(F_0,F_1)}.\]
where $C>0$ does not depend on $T$. If $C=1$, we say that the interpolation pair $(E_\theta,F_\theta)$ is \textit{exact}.

An \textit{(exact) interpolation space} for a Banach couple $(E_0,E_1)$ is an intermediate space $E_\theta$ for $(E_0,E_1)$ such that $(E_\theta,E_\theta)$ is an (exact) interpolation pair for $(E_0,E_1)$, $(E_0,E_1)$.

An \textit{(exact) interpolation functor} $\mathcal{F}$ associates to each Banach couple $(E_0,E_1)$ an intermediate space $\mathcal{F}(E_0,E_1)$, such that if $(E_0,E_1)$ and $(F_0,F_1)$ are two Banach couples, then $(\mathcal{F}(E_0,E_1),\mathcal{F}(F_0,F_1))$ is an (exact) interpolation pair for $(E_0,E_1)$, $(E_0,E_1)$.

For instance, the map that associates to each Banach couple $(E_0,E_1)$ the sum space $E_0+E_1$ (resp. The intersection space $E_0\cap E_1$, the left space $E_0$, the right space $E_1$) is an exact interpolation functor. 

Let $(B_0,B_1)$ be a \textit{subcouple} of a Banach couple $(A_0,A_1)$, which means that $B_0,B_1$ are Banach spaces together with continuous, linear, and with closed range inclusions $B_0\subset A_0, B_1\subset A_1$, that allow us to see $(B_0,B_1)$ as a Banach couple. Then we have a continuous and linear inclusion $\mathcal{F}(B_0,B_1)\subset\mathcal{F}(A_0,A_1)$, but not necessarily with closed range. We say that
$(B_0,B_1)$ is \textit{complemented} in $(A_0,A_1)$ if there is a bounded operator $P:(A_0,A_1)\to(B_0,B_1)$ that restricts to the identity on $B_0+B_1$. If $(B_0,B_1)$ is complemented in $(A_0,A_1)$, then for every interpolation functor $\mathcal{F}$, we have
\[\mathcal{F}(B_0,B_1)=(B_0+B_1)\cap\mathcal{F}(A_0,A_1),\]
and the norm of $\mathcal{F}(A_0,A_1)$ induces an equivalent norm in $\mathcal{F}(B_0,B_1)$ (in other words, $\mathcal{F}(B_0,B_1)$ is closed in $\mathcal{F}(A_0,A_1)$).

A Banach couple $(E_0,E_1)$ is \textit{regular} if $E_0\cap E_1$ is dense in both $E_0$ and $E_1$. If $(E_0,E_1)$ is a regular couple, then we have continuous linear inclusions $E_0^*\subset(E_0\cap E_1)^*$ and $E_1^*\subset(E_0\cap E_1)^*$, which allow us to see $(E_0^*,E_1^*)$ as a Banach couple. 
  
\subsubsection{K-method interpolation spaces} If $(E_0,E_1)$ is a Banach couple and $u\in E_0+E_1$, we define its \textit{$K$-functional} by setting for $t>0$, 
\[K_t(u)=K_t(u,E_0,E_1):=\inf\big\{\|u_0\|_{E_0}+t\|u_1\|_{E_1}\ \mid\ u_0\in E_0,\ u_1\in E_1,\ u=u_0+u_1\big\}.\]
Note that the function $t\mapsto K_t(u)$ is concave (and, in particular, Lebesgue measurable). Moreover, for every $t>0$, $K_t$ gives an equivalent norm on $E_0+E_1$. If $(E_0,E_1)$ and $(F_0,F_1)$ are two Banach couples and $T:(E_0,E_1)\to(F_0,F_1)$ a bounded operator, then for every $u\in E_0+E_1$ and $t>0$, we have $K_t(Tu,F_0,F_1)\pp\|T\|_{(E_0,E_1)\to(F_0,F_1)}K_t(u,E_0,E_1)$. In particular, if $(B_0,B_1)$ is a subcouple of a Banach couple $(A_0,A_1)$, then for every $b\in B_0+B_1$ and $t>0$, we have $K_t(b,A_0,A_1)\pp CK_t(b,B_0,B_1)$, where $C>0$ denotes the maximum between the norms of the two inclusions $B_0\to A_0, B_1\to A_1$.

The following concept on $K$-functionals was formally introduced by Pisier in \cite{PisierHardy} and will be essential in this paper. The terminology follows \cite{Narcisse2023}.

\begin{defi}
A subcouple $(B_0,B_1)$ of a Banach couple $(A_0,A_1)$ is \textit{$K$-closed} if there is a constant $C>0$ such that for every $b\in B_0+B_1$ and $t>0$, $K_t(b,B_0,B_1)\pp CK_t(b,A_0,A_1)$.
\end{defi}

$K$-closedness is intimately connected to the following property.

\begin{defi}
A subcouple $(B_0,B_1)$ of a Banach couple $(A_0,A_1)$ is \textit{$K$-complemented} if there is a constant $C>0$ such for every $b\in B_0+B_1$, whenever $b=a_0+a_1$ with $a_0\in A_0$, $a_1\in B_1$, then $b=b_0+b_1$ with $b_0\in B_0$, $b_1\in B_1$ and $\|b_0\|_{B_0}\pp C\|a_0\|_{A_0}$, $\|b_1\|_{B_1}\pp C\|a_1\|_{A_1}$.
\end{defi}

The definition of $K$-closedness used in \cite{Kislyakov} matches the above definition. The following proposition explains how the two definitions are related.

\begin{prop}
Let $(B_0,B_1)$ be a subcouple of a Banach couple $(A_0,A_1)$. The following statements are equivalent.
\begin{enumerate}
    \item[{\rm1)}] $(B_0,B_1)$ is $K$-complemented in $(A_0,A_1)$.
    \item[{\rm2)}] $(B_0,B_1)$ is $K$-closed in $(A_0,A_1)$ and $B_j=(B_0+B_1)\cap A_j$ for $j=0,1$.
\end{enumerate}
\end{prop}
\begin{proof}
The implication $1)\Rightarrow2)$ is straightforward. Assume that $2)$ holds with constant $C>0$. Let $b\in B_0+B_1$ such that $b=a_0+a_1$ with $a_0\in A_0$, $a_1\in A_1$. If $a_0=0$ or $a_1=0$, then $b\in(B_0+B_1)\cap A_0=B_0$ or $b\in(B_0+B_1)\cap A_1=B_1$, and the conclusion is trivial, so we can assume that $a_0\neq0,a_1\neq0$. Then we can use the $K$-closedness assumption with $t=\|a_0\|_{A_0}/\|a_1\|_{A_1}$, which gives $b_0\in B_0$, $b_1\in B_1$ such that $b=b_0+b_1$ and $\|b_0\|_{B_0}+t\|b_1\|_{B_1}\pp C(\|a_0\|_{A_0}+t\|a_1\|_{A_1})=2C\|a_0\|_{A_0}$. Hence $\|b_0\|_{B_0}\pp 2C\|a_0\|_{A_0}$ and $\|b_1\|_{B_1}\pp 2C\|a_0\|_{A_0}/t=2C\|a_1\|_{A_1}$.
\end{proof}

\begin{rem}
If $(B_0,B_1)$ is a complemented subcouple of a Banach couple $(A_0,A_1)$, then it is obviously $K$-complemented.     
\end{rem}

The following useful result is due to Pisier. It will be used several times in the sequel. A complete proof can be found in \cite[Theorem 6.1]{JansonInterpolation}.

\begin{theo}[Pisier]\label{PisierLemma}
Let $(A_0,A_1)$ be a regular Banach couple and $(B_0,B_1)$ be a regular subcouple of $(A_0,A_1)$. Then the following statements are equivalent.
\begin{enumerate}
    \item[{\rm1)}] $(B_0,B_1)$ is $K$-closed in $(A_0,A_1)$.
    \item[{\rm2)}] $(B_0^{\bot},B_1^{\bot})$ is $K$-closed in $(A_0^*,A_1^*)$.
\end{enumerate}
\end{theo}

Now we turn to the construction of a family of exact interpolation functors arising from $K$-functionals. A \textit{$K$-method parameter} is a Banach space $\Phi$ of Lebesgue measurable functions on $\R_+^*$ that contains the function $t\mapsto 1\wedge t$ and such that if $f,g\in\Phi$ with $|g|\pp|f|$ then $\|g\|_\Phi\pp\|f\|_\Phi$. If $\Phi$ is a $K$-method parameter, then for every Banach couple $(E_0,E_1)$, the set
\[K_\Phi(E_0,E_1):=\big\{u\in E_0+E_1\ \mid\ t\mapsto K_t(u)\in\Phi\big\}\]
is a subspace of $E_0+E_1$ and a Banach space under the norm $\|\cdot\|_{K_\Phi(E_0,E_1)}$ defined as follows,
\[\|u\|_{K_\Phi(E_0,E_1)}:=\|t\mapsto K_t(u)\|_{\Phi}.\]
This yields an exact interpolation functor $K_\Phi$ called the \textit{$K$-method} with parameter $\Phi$. 

\begin{prop}\label{Kclosedness}
Let $(B_0,B_1)$ be a $K$-closed subcouple of a Banach couple $(A_0,A_1)$ with constant $C>0$. Then for every $K$-method parameter $\Phi$, we have
\[K_\Phi(B_0,B_1)=(B_0+B_1)\cap K_\Phi(A_0,A_1),\]
and the norm of $K_\Phi(A_0,A_1)$ induces an equivalent norm on $K_\Phi(B_0,B_1)$ (in other words, $K_\Phi(B_0,B_1)$ is closed in $K_\Phi(A_0,A_1)$).
\end{prop}

\subsubsection{Real interpolation spaces} Let $(E_0,E_1)$ be a Banach couple. If $0<\theta<1$ and $1\pp p\pp\infty$, the \textit{real interpolation space} $(E_0,E_1)_{\theta,p}$ is the interpolation space $\Phi_{\theta,p}(E_0,E_1)$, where $\Phi_{\theta,p}$ denotes the $K$-method parameter such that for every Lebesgue-measurable function $f$ on $\R_+^*$, we have $f\in\Phi_{\theta,p}$ if and only if the function $t\mapsto t^{-\theta}f(t)$ belongs to the usual Lebesgue space $L_p(\R_+^*)$ where $\R_+^*$ is equipped with its multiplicative Haar measure $t^{-1}dt$, in which case we have $\|f\|_{\Phi_{\theta,p}}=\|t\mapsto t^{-\theta}f(t)\|_{L_p(\R_+^*)}$. We recall the reiteration theorem for the real interpolation method due to Lions and Peetre. A complete proof can be found in \cite[Theorem 3.1]{Holmstedt}.

\begin{theo}\label{Peetre}
Let $(E_0,E_1)$ be a Banach couple. We set
\[E_{\theta_0}:=(E_0,E_1)_{\theta_0,p_0}\ \ \ \ \ \text{and}\ \ \ \ \ E_{\theta_1}:=(E_0,E_1)_{\theta_1,p_1},\]
where $0\pp\theta_0<\theta_1\pp1$ and $1\pp p_0,p_1\pp\infty$ (with the convention $E_{\theta_0}=E_0$ if $\theta_0=0$ and $E_{\theta_1}=E_1$ if $\theta_1=1$). Let $0<\lambda<1$ and $1\pp p\pp\infty$. Then,
\[(E_{\theta_0},E_{\theta_1})_{\lambda,p}=(E_0,E_1)_{\theta_\lambda,p}\]
with equivalent norms, and where $\theta_\lambda:=(1-\lambda)\theta_0+\lambda\theta_1$.
\end{theo}

In the sequel, by a nontrivial interval, we will mean an interval which is not empty, nor a singleton. From the above theorem we easily deduce the following corollary.

\begin{coro}\label{CoroPeetre}
Let $I_0=[p_0,q_0]\subset[1,\infty]$ be a nontrivial closed interval and $(E_p)_{p\in I_0}$ be a Banach family. We assume that for every $0<\theta<1$, we have $(E_{p_0},E_{q_0})_{\theta,r}=E_r$ with equivalent norms, where $1/r=(1-\theta)/p_0+\theta/q_0$. Then, for every $p,q\in I_0$, $p\neq q$ and $0<\theta<1$, we have $(E_p,E_q)_{\theta,r}=E_r$ with equivalent norms, where $1/r=(1-\theta)/p+\theta/q$.  
\end{coro}

We now recall Wolff's result on the real interpolation method, which can be interpreted as a converse to the reiteration theorem. A complete proof can be found in \cite{Wolff}.

\begin{theo}[Wolff]\label{Wolff}
Let $(E_0,E_{\theta_0},E_{\theta_1},E_1)$ be a Banach family. Assume that
\[E_{\theta_0}:=(E_0,E_{\theta_1})_{\eta_0,p_0}\ \ \ \ \ \text{and}\ \ \ \ \ E_{\theta_1}:=(E_{\theta_0},E_1)_{\eta_1,p_1},\]
with equivalent norms, where $0<\eta_0,\eta_1<1$ and $1\pp p_0,p_1\pp\infty$. Then,
\[E_{\theta_0}=(E_0,E_1)_{\theta_0,p_0}\ \ \ \ \ \text{and}\ \ \ \ \ E_{\theta_1}=(E_0,E_1)_{\theta_1,p_1}\]
with equivalent norms, where $\theta_0:=\frac{\eta_0\eta_1}{1-\eta_0+\eta_0\eta_1}$ and $\theta_1:=\frac{\eta_1}{1-\eta_0+\eta_0\eta_1}$ are determined by the relations $\theta_1=(1-\eta_1)\theta_0+\eta_1$ and $\theta_0=\eta_0\theta_1$.
\end{theo}

\begin{coro}\label{CoroWolff}
Let $I_0=[p_0,q_0]\subset[1,\infty]$ be a nontrivial closed interval and $(E_p)_{p\in I_0}$ be a Banach family. We make the following two assumptions:
\begin{enumerate}
    \item[{\rm(i)}] $I_0=I_1\cup\ldots\cup I_n$ where $I_1,\ldots,I_n\subset[1,\infty]$ are nontrivial closed intervals such that $I_k\cap I_{k+1}$ is nontrivial for every $k\in\{1,\ldots,n-1\}$. 
    \item[{\rm(ii)}] Let $k\in\{1,\ldots,n\}$ and let $p_k,q_k\in I_k$ be such that $I_k=[p_k,q_k]$. Then for every $0<\theta<1$, we have $(E_{p_k},E_{q_k})_{\theta,r}=E_r$ with equivalent norms and where $1/r=(1-\theta)/p_k+\theta/q_k$.
\end{enumerate}
Then for every $0<\theta<1$, we have $(E_{p_0},E_{q_0})_{\theta,p}=E_r$ with equivalent norms, where $1/r=(1-\theta)/p_0+\theta/q_0$. 
\end{coro}

We have a reiteration theorem for $K$-functionals, which is a consequence of Holmstedt's formula (see \cite[Theorem 2.1]{Holmstedt}).

\begin{theo}\label{Holmstedt}
Let $(B_0,B_1)$ be a $K$-closed subcouple of a Banach couple $(A_0,A_1)$. We set
\[A_{\theta_0}:=(A_0,A_1)_{\theta_0,p_0}\ \ \ \ \ \text{and}\ \ \ \ \ A_{\theta_1}:=(A_0,A_1)_{\theta_1,p_1},\]
where $0\pp\theta_0<\theta_1\pp1$ and $1\pp p_0,p_1\pp\infty$ (with the convention $A_{\theta_0}=A_0$ if $\theta_0=0$ and $A_{\theta_1}=E_1$ if $\theta_1=1$) and adopt analogous notations for $(B_0,B_1)$. Then $(B_{\theta_0},B_{\theta_1})$ is $K$-closed in $(A_{\theta_0},A_{\theta_1})$. 
\end{theo}

We easily deduce the following corollary.

\begin{coro}\label{CoroHolmstedt}
Let $I_0=[p_0,q_0]\subset[1,\infty]$ be a nontrivial closed interval and $(B_p)_{p\in I_0}$ be a subfamily of a Banach family $(A_p)_{p\in I_0}$. We make the following two assumptions:
\begin{enumerate}
    \item[{\rm(i)}] For every $0<\theta<1$, we have $(A_{p_0},A_{q_0})_{\theta,p}=A_r$ with equivalent norms, where $1/r=(1-\theta)/p_0+\theta/q_0$.
    \item[{\rm(ii)}] The subcouple $(B_{p_0},B_{q_0})$ is $K$-closed in $(A_{p_0},A_{q_0})$ and for every $0<\theta<1$, we have $(B_{p_0},B_{q_0})_{\theta,p}=B_r$ with equivalent norms, where $1/r=(1-\theta)/p_0+\theta/q_0$.
\end{enumerate}
Then, for every $p,q\in I_0$, the subcouple $(B_p,B_q)$ is $K$-closed in $(A_p,A_q)$. 
\end{coro}

We also have a Wolff-type result for $K$-complementation. A complete proof can be found in \cite[Theorem 2]{Kislyakov}.

\begin{theo}[Kislyakov, Xu]\label{WolffK}
Let $(B_0,B_{\theta_0},B_{\theta_1},B_1)$ be a subfamily of a Banach family $(A_0,A_{\theta_0},A_{\theta_1},A_1)$ such that $(B_0,B_{\theta_1})$ is $K$-complemented in $(A_0,A_{\theta_1})$ and $(B_{\theta_0},B_1)$ is $K$-complemented in $(A_{\theta_0},A_1)$. In addition, assume that
\[A_{\theta_0}:=(A_0,A_{\theta_1})_{\eta_0,p_0}\ \ \ \ \ \text{and}\ \ \ \ \ A_{\theta_1}:=(A_{\theta_0},A_1)_{\eta_1,p_1},\]
\[B_{\theta_0}:=(B_0,B_{\theta_1})_{\eta_0,p_0}\ \ \ \ \ \text{and}\ \ \ \ \ B_{\theta_1}:=(B_{\theta_0},B_1)_{\eta_1,p_1}\]
with equivalent norms, where $0<\eta_0,\eta_1<1$ and $1\pp p_0,p_1\pp\infty$. Then $(B_0,B_1)$ is $K$-complemented in $(A_0,A_1)$.
\end{theo}

By combining Theorem \ref{Wolff}, Theorem \ref{CoroHolmstedt} and Theorem \ref{WolffK} we easily deduce the following corollary, which will be used several times in this paper.

\begin{coro}\label{CoroWolffK}
Let $I_0=[p_0,q_0]\subset[1,\infty]$ be a nontrivial closed interval and $(B_p)_{p\in I_0}$ be a subfamily of a Banach family $(A_p)_{p\in I_0}$. We make the following four assumptions:
\begin{enumerate}
    \item[{\rm(i)}] For every $0<\theta<1$, we have $(A_{p_0},A_{q_0})_{\theta,p}=A_r$ with equivalent norms, where $1/r=(1-\theta)/p_0+\theta/q_0$. \item[{\rm(ii)}] For every $p,q\in I$, $B_p=(B_p+B_q)\cap A_p$.
    \item[{\rm(iii)}] $I_0=I_1\cup\ldots\cup I_n$ where $I_1,\ldots,I_n\subset[1,\infty]$ are closed nontrivial intervals such that $I_k\cap I_{k+1}$ is nontrivial for every $k\in\{1,\ldots,n-1\}$. 
    \item[{\rm(iv)}] Let $k\in\{1,\ldots,n\}$ and let $p_k,q_k\in I_k$ be such that $I_k=[p_k,q_k]$. Then the subcouple $(B_{p_k},B_{q_k})$ is $K$-closed in $(A_{p_k},A_{q_k})$ and for all $0<\theta<1$, we have $B_r=(B_{p_k},B_{q_k})_{\theta,r}$ with equivalent norms, where $1/r=(1-\theta)/p_k+\theta/q_k$. 
\end{enumerate}
Then, for every $p,q\in I_0$, the subcouple $(B_p,B_q)$ is $K$-complemented in $(A_p,A_q)$.  
\end{coro}

\subsection{Noncommutative Lebesgue spaces} We refer to \cite{DoddsSukochevIntegration} for all the details on noncommutative Lebesgue spaces. Let $\mathcal{M}$ be a von Neumann algebra equipped with a normal semifinite faithful (n.s.f.) trace $\tau$. The unit of $\mathcal{M}$ is denoted by $1$, and the underlying Hilbert space on which $\mathcal{M}$ acts is denoted by $\mathcal{H}$. A closed and densely defined operator $x$ on $\mathcal{H}$ with polar decomposition $x=u|x|$ is \textit{affiliated} with $\mathcal{M}$ if the spectral projections $1_{(\lambda,\infty)}(|x|)$, $\lambda>0$, as well as $u$, belong to $\mathcal{M}$, and in that case its \textit{generalized singular value function} (or \textit{decreasing rearrangement function}) is the function on $\R_+^*$ denoted $x^\sharp$ such that
\[x^\sharp(s):=\inf\big\{\lambda>0\ \mid\ \tau(1_{(\lambda,\infty)}(|x|))\pp s\big\},\ \ \ \ \ \ \ \ \ \ s>0.\]
A closed and densely defined operator $x$ on $\mathcal{H}$ is \textit{$\tau$-measurable} if it is affiliated with $\mathcal{M}$ and if its decreasing rearrangement function takes finite values. The set of all $\tau$ measurable operators on $\mathcal{H}$ will be denoted by $L_0(\mathcal{M})$. Then $L_0(\mathcal{M})$ inherits a complete Hausdorff topological $*$-algebra structure, contains $\mathcal{M}$ as a dense $\ast$-subalgebra, and $\tau$ is canonically extended to the positive part of $L_0(\mathcal{M})$. For every $x\in L_0(M)$ and $1\pp p\pp\infty$, we set 
\[\|x\|_p:=\left\{\begin{array}{cl}
    \tau(|x|^p)^{1/p} & {\rm if}\ p<\infty \\
    \inf\{s>0\ \mid\ \tau(1_{(s,\infty)}(|x|))=0\} & {\rm if}\ p=\infty
\end{array}\right..\]
Then, for every $1\pp p\pp\infty$, the associated \textit{noncommutative Lebesgue space},
\[L_p(\mathcal{M}):=\big\{x\in L_0(\mathcal{M})\ \mid\ \|x\|_p<\infty\big\}\]
is a Banach space under $\|\cdot\|_p$ and the inclusion $L_p(\mathcal{M})\subset L_0(\mathcal{M})$ is continuous for $\|\cdot\|_p$. Moreover, the trace $\tau$ extends to a positive and contractive linear form on $L_1(\mathcal{M})$, and we have
\[L_\infty(\mathcal{M})=\mathcal{M}\]
with equal norms. An important feature of noncommutative Lebesgue spaces is the H\"older's inequality. That is, if $1\pp p,q\pp\infty$ are H\"older conjugate exponents, then for all $x,y\in L_0(\mathcal{M})$, we have
\[\|xy\|_1\pp\|x\|_{p}\|y\|_{q}.\]
Moreover, if $x\in L_0(\mathcal{M})$ then
\[\|x\|_p=\sup_{\|y\|_{q}\pp1}\|xy\|_1.\]
From the above results can be derived from trace duality for noncommutative Lebesgue spaces. That is, trace $\tau$ induces a canonical linear isometry $L_p(\mathcal{M})\to L_{p^*}(\mathcal{M})^*$ that is surjective if $p>1$ (or $q<\infty$) and where $1\pp p^*\pp\infty$ is the conjugate exponent of $p$, that is, $1=\frac{1}{p}+\frac{1}{p^*}$. More generally, for every $1\pp p,q\pp\infty$ with conjugate $1\pp p^*,q^*\pp\infty$, the trace $\tau$ induces a canonical linear isometry $(L_{p^*}\cap L_{q^*})(\mathcal{M})\to(L_p+L_q)(\mathcal{M})^*$ which is surjective whenever $p\neq q$, where
\[(L_p+L_q)(\mathcal{M}):=L_p(\mathcal{M})+L_q(\mathcal{M})\]
is equipped with the sum norm, and
\[(L_{p^*}\cap L_{q^*})(\mathcal{M}):=L_{p^*}(\mathcal{M})\cap L_{q^*}(\mathcal{M})\]
is equipped with the intersection norm.

Now we turn to investigating the properties of the Banach couple $(L_1(\mathcal{M}),L_\infty(\mathcal{M}))$. A particularly nice feature is that one has a concrete expression of the associated $K$-functionals.

\begin{prop}[Holmstedt's formula]
Let $x\in L_0(\mathcal{M})$. Then $x\in (L_1+L_\infty)(\mathcal{M})$ if and only if
for every $t>0$, we have
\[\int_0^{t}x^{\sharp}(s)ds<\infty\]
and in that case we have
\[K_t(x,L_1(\mathcal{M}),L_\infty(\mathcal{M}))=\int_0^{t}x^{\sharp}(s)ds\]
for every $t>0$.
\end{prop}

It is easy to see that if $1\pp p\pp\infty$, the Banach space $L_p(\mathcal{M})$ is an intermediate space for $(L_1(\mathcal{M}),L_\infty(\mathcal{M}))$. Actually, using Holmstedt's formula, one can show the following crucial result, which implies that $L_p(\mathcal{M})$ is an interpolation space for $(L_1(\mathcal{M}),L_\infty(\mathcal{M}))$.

\begin{theo}\label{InterpLp}
For every $0<\theta<1$, we have
\[(L_1(\mathcal{M}),L_\infty(\mathcal{M}))_{\theta,p}=L_p(\mathcal{M})\]
with equivalent norms and constants independent of $\mathcal{M}$, where $1/p=1-\theta$.
\end{theo}

An intermediate space $E(\mathcal{M})$ for $(L_1(\mathcal{M}),L_\infty(\mathcal{M}))$ is \textit{fully symmetric} if whenever $x\in E(\mathcal{M})$ and $y\in (L_1+L_\infty)(\mathcal{M})$ are such that
\[\int_{0}^{t}y^{\sharp}(s)ds\pp\int_{0}^{t}x^{\sharp}(s)ds\]
for all $t>0$, then $y\in E(\mathcal{M})$ with $\|y\|_{E(\mathcal{M})}\pp\|y\|_{E(\mathcal{M})}$. For example, if $1\pp p\pp\infty$ then $L_p(\mathcal{M})$ is fully symmetric. The following deep theorem, proved in \cite{DoddsSukochevIntegration}, implies that an intermediate space $E(\mathcal{M})$ for $(L_1(\mathcal{M}),L_\infty(\mathcal{M}))$ is an exact interpolation space if and only if it is fully symmetric. In particular, $L_p(\mathcal{M})$ is actually an exact interpolation space for $(L_1(\mathcal{M}),L_\infty(\mathcal{M}))$ for every $1\pp p\pp\infty$.

\begin{theo}\label{Dodds}
Let $\mathcal{M},\mathcal{N}$ be a pair of von Neumann algebras equipped with a n.s.f. trace. A pair $E(\mathcal{M})$, $E(\mathcal{N})$ of intermediate spaces for $(L_1(\mathcal{M}),L_\infty(\mathcal{M}))$, $(L_1(\mathcal{N}),L_\infty(\mathcal{N}))$, is an exact interpolation pair if and only if whenever $x\in E(\mathcal{M})$ and $y\in (L_1+L_\infty)(\mathcal{N})$ are such that
\[\int_{0}^{t}y^{\sharp}(s)ds\pp\int_{0}^{t}x^{\sharp}(s)ds\]
for all $t>0$, then $y\in E(\mathcal{N})$ with $\|y\|_{E(\mathcal{N})}\pp\|x\|_{E(\mathcal{M})}$. 
\end{theo}

In the sequel, we shall consider the half line $\R_+^*$ equipped with the Lebesgue measure, and we will denote $L_p:=L_p(\R_+^*)$ for $1\pp p\pp\infty$ the associated Lebesgue spaces. If $E$ is a fully symmetric intermediate space (or an exact interpolation space) for $(L_1,L_\infty)$, then the set
\[E(\mathcal{M}):=\Big\{x\in L_1(\mathcal{M})+L_\infty(\mathcal{M})\ \mid\ x^{\sharp}\in E\Big\}\]
becomes a Banach space when equipped with the norm $\|\cdot\|_{E(\mathcal{M})}$ defined as follows,
\[\|x\|_{E(\mathcal{M})}:=\|x^{\sharp}\|_{E}.\]
Moreover, $E(\mathcal{M})$ is clearly a fully symmetric intermediate space (and thus an exact interpolation space) for $(L_1(\mathcal{M}),L_\infty(\mathcal{M}))$. 

More generally, Theorem \ref{Dodds} yields the following result. 

\begin{prop}\label{Dodds2}
Let $\mathcal{M},\mathcal{N}$ be von Neumann algebras equipped with n.s.f. traces. Then for every fully symmetric intermediate space $E$ for $(L_1,L_\infty)$, the pair $(E(\mathcal{M}),E(\mathcal{N}))$ is an exact interpolation pair for $(L_1(\mathcal{M}),L_\infty(\mathcal{M}))$, $(L_1(\mathcal{N}),L_\infty(\mathcal{N}))$. 
\end{prop}

Notice that with the choice $E:=L_p$ with $1\pp p\pp\infty$, we have $E(\mathcal{M})=L_p(\mathcal{M})$ with equal norms. The following important result, which provides a generalisation of Theorem \ref{InterpLp}, can also be found in \cite{DoddsSukochevIntegration}.

\begin{theo}\label{FunctorLp}
Let $\mathcal{F}$ be an exact interpolation functor and consider the exact interpolation space $E:=\mathcal{F}(L_1,L_\infty)$. Then
\[E(\mathcal{M})=\mathcal{F}(L_1(\mathcal{M}),L_\infty(\mathcal{M}))\]
with the same norms.
\end{theo}

\subsection{Noncommutative martingales} We refer to \cite{PisierXuMartingales} for a background on noncommutative martingales theory. Let $\mathcal{M}\subset\mathcal{B}(\mathcal{H})$ be a von Neumann algebra endowed with a n.s.f. trace $\tau$ and a \textit{filtration}, i.e., an increasing sequence $(\mathcal{M}_n)_{n\pg1}$ of von Neumann subalgebras of $\mathcal{M}$ whose union $\cup_{n\pg1}\mathcal{M}_n$ is weak*-dense in $\mathcal{M}$. We further require that we have a $\tau$-preserving weakly*-continuous conditional expectation $\EE_n$ from $\mathcal{M}$ to $\mathcal{M}_n$. Then $(\EE_n)_{n\pg1}$ is an increasing sequence of commuting projections. For every $n\pg1$, we set
\[\DD_n:=\EE_n-\EE_{n-1}\]
(with the convention $\EE_0=0$). Then $(\DD_n)_{n\pg1}$ is a family of mutually orthogonal projections that commutes with $(\EE_n)_{n\pg1}$. We will refer to them as the \textit{difference projections} associated with the filtration. We recall that $\EE_n$ extends to a contractive projection from $L_p(\mathcal{M})$ onto $L_p(\mathcal{M}_n)$ for every $1\pp p\pp\infty$, and, in particular, it extends to a weakly continuous projection from $(L_1+L_\infty)(\mathcal{M})$ onto $(L_1+L_\infty)(\mathcal{M}_n)$, for every $n\pg1$. A sequence $(x_n)_{n\pg1}$ of $(L_1+L_\infty)(\mathcal{M})$ is \textit{adapted} if $\EE_n(x_n)=x_n$ for all $n\pg1$. 

\begin{defi}
A \textit{martingale} is a sequence $(x_n)_{n\pg1}$ of $(L_1+L_\infty)(\mathcal{M})$ which is adapted and such that $\EE_{n-1}(x_n)=x_{n-1}$ for all $n\pg2$. A martingale $(x_n)_{n\pg1}$ is \textit{finite} if it sits in $(L_1+L_\infty)(\mathcal{M}_n)$ for a certain $n\pg1$. A sequence $(x_n)_{n\pg1}$ of $(L_1+L_\infty)(\mathcal{M})$ is a \textit{martingale difference} if it is adapted and $\EE_{n-1}(x_n)=0$ for all $n\pg2$.
\end{defi}

For every $x\in(L_1+L_\infty)(\mathcal{M})$, the sequence $(\EE_n(x))_{n\pg1}$ is a martingale, and the sequence $(\DD_n(x))_{n\pg1}$ is a martingale difference, and we say they are \textit{induced} by $x$. Note that we have $x\in\cup_{n\pg 1}(L_1+L_\infty)(\mathcal{M}_n)$ if and only if the martingale $(\EE_n(x))_{n\pg1}$ is finite, and also if and only if the martingale difference sequence $(\DD_n(x))_{n\pg1}$ is eventually-zero.

\begin{prop}
Let $1\pp p\pp\infty$. If $x\in L_p(\mathcal{M})$, then the sequence $(\EE_n(x))_{n\pg1}$ converges to $x$ (with respect to the $L_p$-norm if $p<\infty$, and to the weak* topology if $p=\infty$), and $\|x\|_p=\sup_{n\pg1}\|\EE_n(x)\|_p$. 
\end{prop}

\begin{rem}
As a consequence, if $x\in(L_1+L_\infty)(\mathcal{M})$, then the sequence $(\EE_n(x))_{n\pg1}$ converges weakly to $x$ in the Banach space $(L_1+L_\infty)(\mathcal{M})$.
\end{rem}

Finally, we state a noncommutative version of Gundy's decomposition, which will be an important tool to prove the main results of the paper. The following theorem is adapted from \cite{ParcetGundy}[Corollary 2.10].

\begin{theo}\label{ParcetRandria}
Let $y\in\cup_{n\pg1}L_1(\mathcal{M}_n)$ be a positive operator and $\lambda>0$. Then there exist positive operators $\alpha,\beta,\gamma\in\cup_{n\pg1}L_1(\mathcal{M}_n)$, an adapted sequence $(q_n)_{n\pg1}$ of projections of $\mathcal{M}$ and a universal constant $C>0$ with the following properties,
\begin{enumerate}
    \item[{\rm(i)}] $y=\alpha+\beta+\gamma$,
    \item[{\rm(ii)}] the element $\alpha$ satisfies
\[\|\alpha\|_1\pp C\|y\|_1,\ \ \ \ \ \ \ \ \ \ \ \ \ \ \ \|\alpha\|_2^2\pp C^2\lambda\|y\|_1,\]
    \item[{\rm(iii)}] the element $\beta$ satisfies 
\[\sum_{n\pg1}\|\DD_n(\beta)\|_1\pp C\|y\|_1,\]
    \item[{\rm(iv)}] the element $\gamma$ satisfies
\[\DD_{n+1}(\gamma)=\DD_{n+1}(y)-q_n\DD_{n+1}(y)q_n\]
for every $n\pg1$, 
    \item[{\rm(v)}] the projection $p:=1-\wedge_{n\pg1}q_n$ satisfies
\[\tau(p)\pp C\frac{\|y\|_1}{\lambda}.\]
\end{enumerate}
\end{theo}

The following version of noncommutative Gundy's decomposition then easily follows.

\begin{coro}[Gundy's decomposition]
Let $y\in\cup_{n\pg1}L_1(\mathcal{M}_n)$ and $\lambda>0$. Then there exist $\alpha,\beta,\gamma\in\cup_{n\pg1}L_1(\mathcal{M}_n)$ and a universal constant $C>0$ with the following properties,
\begin{enumerate}
    \item[{\rm(i)}] $y=\alpha+\beta+\gamma$,
    \item[{\rm(ii)}] the element $\alpha$ satisfies
\[\|\alpha\|_1\pp C\|y\|_1,\ \ \ \ \ \ \ \ \ \ \ \ \ \ \ \|\alpha\|_2^2\pp C^2\lambda\|y\|_1,\]
    \item[{\rm(iii)}] the element $\beta$ satisfies 
\[\sum_{n\pg1}\|\DD_n(\beta)\|_1\pp C\|y\|_1,\]
    \item[{\rm(iv)}] the element $\gamma$ satisfies
\[(1-p)\DD_n(\gamma)(1-p)=0\]
for every $n\pg1$ where $p\in\mathcal{M}$ is a projection such that
\[\tau(p)\pp C\frac{\|y\|_1}{\lambda}.\]
\end{enumerate}
\end{coro}

\section{Main results}

This section contains statements and proofs of new general $K$-closedness results in noncommutative Lebesgue spaces when filtrations are involved.

The section is organised as follows. In the first paragraph, by extracting the arguments used by Bourgain in \cite{Bourgain} we introduce the notion of Gundy projections. We then investigate some of their properties. We state and prove a general $K$-closedness result that will be used in the final paragraph, and which provides a noncommutative analogue of \cite{Bourgain}[Lemma 2.4]. In the second paragraph, we use the material of Gundy projections to state and prove new $K$-closedness results in noncommutative Lebesgue spaces on von Neumann algebras equipped with a n.s.f. trace and one or several filtrations.  

\subsection{Gundy projections}

Let $\mathcal{M}$ be a von Neumann algebra equipped with a n.s.f. trace $\tau$. In this section, we will be mainly interested in the study of idempotent operators
\[\PP:D\to D\]
whose domain $D$ is a subspace of $(L_1+L_\infty)(\mathcal{M})$, that are \textit{Lebesgue compatible}, i.e such that for every $1\pp p\pp\infty$, $\PP$ induces a (not necessarily bounded) linear idempotent operator on $D\cap L_p(\mathcal{M})$, in the sense that $\PP(D\cap L_p(\mathcal{M}))\subset L_p(\mathcal{M})$. If $\PP:D\to D$ is a Lebesgue compatible linear idempotent operator, then for every $1\pp p\pp\infty$ we will denote by $\PP_p(\mathcal{M})$ the closure of $\PP(D\cap L_p(\mathcal{M}))$ in $L_p(\mathcal{M})$ (with respect to the norm topology if $p<\infty$ and the weak* topology if $p=\infty$).

\begin{defi}
A \textit{Gundy projection} on $\mathcal{M}$ with constant $C>0$ is a Lebesgue compatible linear idempotent operator $\PP:D\to D$, whose domain $D$ is a subspace of $(L_1+L_\infty)(\mathcal{M})$, such that $\PP$ is self-adjoint (or equivalently, contractive) on $D\cap L_2(\mathcal{M})$, and such that for every $y\in D\cap L_1(\mathcal{M})$ and $\lambda>0$, there exist $\alpha,\beta,\gamma\in\PP(D)$ with the following properties,
\begin{enumerate}
    \item[{\rm(i)}] $\PP(y)=\alpha+\beta+\gamma$,
    \item[{\rm(ii)}] the element $\alpha$ satisfies
\[\|\alpha\|_2^2\pp C^2\lambda\|y\|_1,\]
    \item[{\rm(iii)}] the element $\beta$ satisfies 
\[\|\beta\|_1\pp C\|y\|_1,\]
    \item[{\rm(iv)}] the element $\gamma$ satisfies
\[(1-p)\gamma(1-p)=0,\]
where $p\in\mathcal{M}$ is a projection such that
\[\tau(p)\pp C\frac{\|y\|_1}{\lambda}.\]
\end{enumerate}
\end{defi}

\begin{rem}
Notice that if $\PP$ is bounded for the $L_1$-norm, then clearly $\PP$ is a Gundy projection. For non trivial applications, we will need to consider the case where $\PP$ is not bounded for the $L_1$-norm (for instance, operators arising from martingale transforms, see the next paragraph). 
\end{rem}

\begin{prop}\label{Gundy - Proposition 1}
Let $\PP^{(0)}$ and $\PP^{(1)}$ be two Gundy projections on $\mathcal{M}$ with the same domain $D$, with constants $C_0,C_1>0$, respectively, and assume that $\PP^{(0)}$ and $\PP^{(1)}$ are orthogonal, i.e. $\PP^{(0)}\PP^{(1)}=\PP^{(1)}\PP^{(0)}=0$. Then, $\PP:=\PP^{(0)}+\PP^{(1)}$ is a Gundy projection with domain $D$ and constant $C:=C_0+C_1$. 
\end{prop}
\begin{proof}
It is trivial that $\PP$ is Lebesgue compatible and self-adjoint on $D\cap L_2(\mathcal{M})$. Let $y\in D$ and $\lambda>0$. Then, for $j\in\{0,1\}$, there exist $\alpha^{(j)},\beta^{(j)},\gamma^{(j)}\in\PP^{(j)}(D)$ with the following properties,
\begin{enumerate}
    \item[{\rm(i)}] $\PP^{(j)}(y)=\alpha^{(j)}+\beta^{(j)}+\gamma^{(j)}$,
    \item[{\rm(ii)}] the element $\alpha^{(j)}$ satisfies 
\[\|\alpha^{(j)}\|_2^2\pp C_j^2\lambda\|y\|_1,\]
    \item[{\rm(iii)}] the element $\beta$ satisfies 
\[\|\beta^{(j)}\|_1\pp C_j\|y\|_1,\]
    \item[{\rm(iv)}] the element $\gamma$ satisfies
\[(1-p^{(j)})\gamma^{(j)}(1-p^{(j)})=0,\]
where $p^{(j)}\in\mathcal{M}$ is a projection such that
\[\tau(p^{(j)})\pp C_j\frac{\|y\|_1}{\lambda}.\]
\end{enumerate}
We set 
\[\alpha:=\alpha^{(0)}+\alpha^{(1)},\ \ \ \ \ \beta:=\beta^{(0)}+\beta^{(1)},\ \ \ \ \ \gamma:=\gamma^{(0)}+\gamma^{(1)}.\]
Then we have $\alpha,\beta,\gamma\in \PP(D)$, and
\[\PP(y)=\PP^{(0)}(y)+\PP^{(1)}(y)=\alpha+\beta+\gamma.\]
The element $\alpha$ satisfies 
\[\|\alpha\|_2^2\pp(\|\alpha^{(0)}\|_2+\|\alpha^{(1)}\|_2)^2\pp(C_0\sqrt{\lambda\|y\|_1}+C_1\sqrt{\lambda\|y\|_1})^2=C^2\lambda\|y\|_1.\]
The element $\beta$ satisfies 
\[\|\beta\|_1=\|\beta^{(0)}+\beta^{(1)}\|_1\pp\|\beta^{(0)}\|_1+\|\beta^{(1)}\|_1\pp C_0\|y\|_1+C_1\|y\|_1=C\|y\|_1\]
Finally, we set 
\[p:=p^{(0)}\vee p^{(1)}.\]
Then 
\[\tau(p)\pp\tau(p^{(0)})+\tau(p^{(1)})\pp C_0\frac{\|y\|_1}{\lambda}+C_1\frac{\|y\|_1}{\lambda}=C\frac{\|y\|_1}{\lambda},\]
and
\[(1-p)\gamma(1-p)=(1-p)\gamma^{(0)}(1-p)+(1-p)\gamma^{(1)}(1-p)=0.\]
The proof is complete.
\end{proof}

Now we turn to the core result of the paper. It is essentially contained in the following crucial lemma, whose proof is directly inspired by the approach of Bourgain in \cite{Bourgain}[Lemma 2.4].

\begin{lemm}\label{Gundy - Lemma 1}
Let $\PP$ be a Gundy projection on $\mathcal{M}$ with domain $D$ and constant $C>0$. Let $x\in \PP(D)$ and $t>0$. If $y,z\in D$ are such that $x=y+z$ with $\|y\|_1\pp 2$, $\|z\|_2\pp\frac{2}{t}$, then there exist $y',z'\in \PP(D)$ such that $x=y'+z'$ and $\|y'\|_{1}\pp C'$, $\|z'\|_{2}\pp\frac{C'}{t}$ where $C'>0$ is a constant which depends only on $C$.
\end{lemm}
\begin{proof}
Let $y,z\in D$ such that $x=y+z$ with $\|y\|_1\pp 2$, $\|z\|_2\pp\frac{2}{t}$. By applying the definition of Gundy projection with $y$ and $\lambda:=t^{-2}$, there exist $\alpha,\beta,\gamma\in \PP(D)$ with the following properties,
\begin{enumerate}
    \item[{\rm(i)}] $\PP(y)=\alpha+\beta+\gamma$,
    \item[{\rm(ii)}] the element $\alpha$ satisfies 
\[\|\alpha\|_2^2\pp C^2t^{-2}\|y\|_1\pp 2C^2t^{-2},\]
    \item[{\rm(iii)}] the element $\beta$ satisfies 
\[\|\beta\|_1\pp C\|y\|_1\pp 2C,\]
    \item[{\rm(iv)}] the element $\gamma$ satisfies
\[(1-p)\gamma(1-p)=0,\]
where $p\in\mathcal{M}$ is a projection such that
\[\tau(p)\pp Ct^2\|y\|_1\pp 2Ct^2,\]
\end{enumerate}
where $C>0$ is a constant which does not depend on $y$ and $\lambda$. As $x=\PP(x)=\PP(y)+\PP(z)$, we have 
\[x=y'+z',\]
where we denoted
\[y':=\beta+\gamma\in \PP(D),\ \ \ \ \ \ \ \ \ \ \ \ z':=\alpha+\PP(z)\in \PP(D).\]
As $\PP$ is contractive on $D\cap L_2(\mathcal{M})$, we have the estimate
\[\|z'\|_2\pp\|\PP(z)\|_2+\|\alpha\|_2\pp\|z\|_2+\|\alpha\|_2\pp\frac{2}{t}+\frac{\sqrt{2}C}{t}=\frac{2+\sqrt{2}C}{t}.\]
To deal with $y'$, we decompose it in three terms as follows,
\[y'=y'_1+y'_2+y'_3,\]
with 
\[y'_1:=(1-p)y'(1-p),\ \ \ \ \ \ \ \ y_2':=py',\ \ \ \ \ \ \ \ y_3':=(1-p)y'p.\]
By using the fact that $(1-p)\gamma(1-p)=0$, we have 
\begin{equation}\label{Gundy - est1}
\|y'_1\|_1=\|(1-p)(\beta+\gamma)(1-p)\|_1=\|(1-p)\beta(1-p)\|_1\pp\|\beta\|_1\pp 2C.
\end{equation}
By writing $y'=y+(z-z')$, and then by using H\"older's inequality we have
\begin{align}
\|y_2'\|_1&=\|py'\|_1\nonumber\\
&\pp\|py\|_1+\|p(z-z')\|_1\nonumber\\
&\pp\|y\|_1+\tau(p)^{1/2}\|z-z'\|_2\nonumber\\
&\pp\|y\|_1+\tau(p)^{1/2}(\|z\|_2+\|z'\|_2)\nonumber\\
&\pp2+\tau(p)^{1/2}\big(\frac{2}{t}+\frac{2+\sqrt{2}C}{t}\big)\nonumber\\
&\pp2+\sqrt{2C}t\big(\frac{2}{t}+\frac{2+\sqrt{2}C}{t}\big)\nonumber\\
&=2+\sqrt{2C}(4+\sqrt{2}C)=:C'.\label{Gundy - est2}
\end{align}
By an analogous computation, we also have $\|y'p\|_1\pp C'$, and therefore
\begin{equation}\label{Gundy - est3}
\|y_3'\|_1=\|(1-p)y'p\|_1\pp\|y'p\|_1\pp C'.
\end{equation}
By combining the estimates \eqref{Gundy - est1}, \eqref{Gundy - est2} and \eqref{Gundy - est3} with the triangle inequality we obtain 
\[\|y'\|_1\pp 2C+2C'.\]
The proof is complete.
\end{proof}

The main result of this paragraph is the following.

\begin{theo}\label{Gundy - Theorem 1}
Let $\PP$ be a Gundy projection on $\mathcal{M}$ with constant $C>0$. Assume that there is a filtration $(\mathcal{M}_n)_{n\pg1}$ on $\mathcal{M}$ with associated conditional expectations denoted $(\EE_n)_{n\pg1}$ such that
\begin{enumerate}
    \item[{\rm a)}] $D=\cup_{n\pg1}(L_1+L_\infty)(\mathcal{M}_n)$,
    \item[{\rm b)}] $\EE_n(x)\in\PP(D)$ for every $1\pp p\pp\infty$, $x\in\PP_p(\mathcal{M})$ and $n\pg1$.
\end{enumerate}
Then $(\PP_1(\mathcal{M}),\PP_2(\mathcal{M}))$ is $K$-closed in $(L_1(\mathcal{M}),L_2(\mathcal{M}))$ with a constant which depends only on $C$.
\end{theo}
\begin{proof}
Let $t>0$ and denote $k_t$, $K_t$ the associated $K$-functionals for $(\PP_1(\mathcal{M}),\PP_2(\mathcal{M}))$ and $(L_1(\mathcal{M}),L_2(\mathcal{M}))$, respectively. It suffices to show that for every $x\in \PP_1(\mathcal{M})+\PP_2(\mathcal{M})$ with $K_t(x)\pp1$, we have $k_t(x)\pp C'$ for a constant $C'>0$ which depends only on $C$. 

\noindent$\triangleright$ Let $x\in \PP(D)\cap(L_1+L_2)(\mathcal{M})$ such that $K_t(x)\pp1$. Then we can write $x=y+z$ with $y\in L_1(\mathcal{M})$, $z\in L_2(\mathcal{M})$ such that $\|y\|_1+t\|z\|_2\pp 2$. As $x\in D$, we have $\EE_n(x)=x$ for some integer $n\pg1$. Thus, we can decompose
\[x=y'+z',\]
where we denoted 
\[y':=\EE_n(y)\in D,\ \ \ \ \ \ \ \ \ z':=\EE_n(z)\in D.\]
As $\|y'\|_1=\|\EE_n(y)\|_1\pp\|y\|_1\pp 2$ and $\|z'\|_2=\|\EE_n(z)\|_2\pp\|z\|_2\pp\frac{2}{t}$, we can apply Lemma \ref{Gundy - Lemma 1} to the above decomposition. Thus, there exist $y'',z''\in\PP(D)$ such that $x=y''+z''$ and $\|y''\|_{1}\pp C'$, $\|z''\|_{2}\pp\frac{C'}{t}$ where $C'\pg$ is a constant which depends only on $C$. We deduce $k_t(x)\pp\|y''\|_1+t\|z''\|_2\pp 2C'$. 

\noindent$\triangleright$ Now, let $x\in \PP_1(\mathcal{M})+\PP_2(\mathcal{M})$ such that $K_t(x)\pp1$. If $n\pg1$, then by assumption b) we know that $\EE_n(x)$ belongs to $\PP(D)\cap(L_1+L_2)(\mathcal{M})$, and we also have $K_t(\EE_n(x))\pp K_t(x)\pp 1$ because $\EE_n$ is contractive on $L_1(\mathcal{M})$ and $L_2(\mathcal{M})$. By what we have just proved, we get $k_t(\EE_n(x))\pp 2C'$. As the sequence $(\EE_n(x))_{n\pg1}$ clearly converges to $x$ in $\PP_1(\mathcal{M})+\PP_2(\mathcal{M})$ for the sum norm, letting $n\to\infty$ in the last estimation yields $k_t(x)\pp 2C'$ as desired. 
\end{proof}

The following corollary will be used in the sequel.

\begin{coro}\label{Gundy - Corollary 1}
Let $\PP$ be a Gundy projection on $\mathcal{M}$ with constant $C>0$. Assume that there is a filtration $(\mathcal{M}_n)_{n\pg1}$ on $\mathcal{M}$ with associated conditional expectations denoted $(\EE_n)_{n\pg1}$ such that
\begin{enumerate}
    \item[{\rm a)}] $D=\cup_{n\pg1}(L_1+L_\infty)(\mathcal{M}_n)$,
    \item[{\rm b)}] $\PP$ is weakly continuous on $(L_1+L_\infty)(\mathcal{M}_n)$ for every $n\pg1$,
    \item[{\rm c)}] $\EE_n$ stabilizes $\PP(D)$ for every $n\pg1$, i.e $\EE_n(x)\in\PP(D)$ when $x\in\PP(D)$.
\end{enumerate}
Then, $\EE_n(x)\in\PP(D)$ for every $1\pp p\pp\infty$, $x\in\PP_p(\mathcal{M})$ and $n\pg1$. As an immediate consequence, $(\PP_1(\mathcal{M}),\PP_2(\mathcal{M}))$ is $K$-closed in $(L_1(\mathcal{M}),L_2(\mathcal{M}))$ with a constant that depends only on $C$.
\end{coro}
\begin{proof}
Let $1\pp p\pp\infty$, $x\in\PP_p(\mathcal{M})$ and $n\pg1$. By definition, there is a net $(x_\alpha)$ of $\PP(D\cap L_p(\mathcal{M}))$ that converges to $x$ weakly in $(L_1+L_\infty)(\mathcal{M})$. Then, $(\EE_n(x_\alpha))$ converges to $\EE_n(x)$ weakly in $(L_1+L_\infty)(\mathcal{M}_n)$. By assumption b), we deduce that $(\PP\EE_n(x_\alpha))$ converges to $\PP\EE_n(x)$ weakly in $(L_1+L_\infty)(\mathcal{M}_n)$. But by assumption c) we have $\PP\EE_n(x_\alpha)=\EE_n(x_\alpha)$ for every $\alpha$. Thus, $\PP\EE_n(x)=\EE_n(x)$ as desired.
\end{proof}

\subsection{Three examples}  

\subsubsection{Example 1}\label{Example1} Let $\mathcal{M}$ be a Neumann algebra equipped with a n.s.f. trace $\tau$ and also equipped with a filtration $(\mathcal{M}_n)_{n\pg1}$ with associated conditional expectations denoted $(\EE_n)_{n\pg1}$ and associated difference projections denoted $(\DD_n)_{n\pg1}$. Let
\[D:=\cup_{n\pg1}(L_1+L_\infty)(\mathcal{M}_n)\]
be the subspace of $(L_1+L_\infty)(\mathcal{M})$ whose elements induce finite martingales with respect to the filtration $(\mathcal{M}_n)_{n\pg1}$. Finally, let $I$ be a set of positive integers and let $\PP:D\to D$ be the Lebesgue compatible linear idempotent operator given by the following expression,
\[\PP(x)=\sum_{i\in I}\DD_i(x),\ \ \ \ \ \ \ x\in D.\]
Note that by the very definition of $D$, the sum appearing in the above expression is always finite so that $\PP$ is well defined. All the following results are only interesting when $I$ is infinite (when $I$ is finite, we notice that $\PP$ is bounded for the $L_p$-norm for all $1\pp p\pp\infty$ with constant less than $|I|$). We recall that if $1\pp p\pp\infty$, then $\PP_p(\mathcal{M})$ denotes the closure of $\PP(D\cap L_p(\mathcal{M}))$ in $L_p(\mathcal{M})$ with respect to the norm topology if $p<\infty$, and to the weak* topology if $p=\infty$. Then we have the following proposition. 

\begin{prop}\label{Example 1 - Proposition 1}
Let $1\pp p\pp\infty$. Then 
\[\PP_p(\mathcal{M})=\big\{x\in L_p(\mathcal{M})\ \mid\ \DD_i(x)=0\ \forall i\notin I\big\}.\]
\end{prop}
\begin{proof}
The space on the right-hand side clearly contains $\PP(D\cap L_p(\mathcal{M}))$ and is closed in $L_p(\mathcal{M})$ (with respect to the norm topology if $p<\infty$ and to the weak* topology if $p=\infty$), from which we deduce the first inclusion. For the converse inclusion, let $x\in L_p(\mathcal{M})$ such that $\DD_i(x)=0$ for all $i\notin I$. Then for $n\pg1$, we have
\[\EE_n(x)=\sum_{k=1}^{n}\DD_k(x)=\sum_{\protect\substack{k=1\\k\in I}}^{n}\DD_k(x).\]
This shows that $\EE_n(x)\in\PP(D\cap L_p(\mathcal{M}))$. As $(\EE_n(x))_{n\pg1}$ converges to $x$ (with respect to the norm topology if $p<\infty$ and to the weak* topology if $p=\infty$), by the definition of $\PP_p(\mathcal{M})$ we deduce $x\in\PP_p(\mathcal{M})$. The proof is complete.
\end{proof}

We are interested in the interpolation properties of the spaces $\PP_p(\mathcal{M})$, $1\pp p\pp\infty$ in terms of their $K$-functionals. The main result is the following. It can be considered as an extension to the range $[1,\infty]$ of the boundedness properties of noncommutative martingale transforms as established by Randrianantoanina in \cite{NarcisseMartingaleTransforms}.

\begin{theo}\label{Example 1 - Theorem 1}
Let $1\pp p,q\pp\infty$. Then the Banach couple $(\PP_p(\mathcal{M}),\PP_q(\mathcal{M}))$ is $K$-complemented in $(L_p(\mathcal{M}),L_q(\mathcal{M}))$ with a universal constant.
\end{theo}

Using the real interpolation machinery introduced in the preliminary section, and, in particular, Corollary \ref{CoroWolffK}, the above Theorem follows from Facts 1-5 below.

\begin{fact}\label{Example 1 - Fact 1}
For every $1\pp p,q\pp\infty$, we have $(\PP_p(\mathcal{M})+\PP_q(\mathcal{M}))\cap L_p(\mathcal{M})=\PP_p(\mathcal{M})$.
\end{fact}
\begin{proof}
This follows trivially from Proposition \ref{Example 1 - Proposition 1}.
\end{proof}

\begin{fact}\label{Example 1 - Fact 2}
$(\PP_1(\mathcal{\mathcal{M}}),\PP_2(\mathcal{\mathcal{M}}))$ is $K$-closed $(L_1(\mathcal{\mathcal{M}}),L_2(\mathcal{\mathcal{M}}))$ with a universal constant.
\end{fact}

The proof of this fact needs the ingredients introduced in the previous paragraph.

\begin{lemm}\label{Example 1 - Lemma 1}
$\PP$ is a Gundy projection on $\mathcal{M}$ with a universal constant.
\end{lemm}
\begin{proof}[Proof of Lemma \ref{Example 1 - Lemma 1}]
It is clear that $\PP$ is self-adjoint on $L_2(\mathcal{M})$. Let $y\in D$ and $\lambda>0$. By applying Gundy's decomposition theorem to $y$ and $\lambda$ with respect to the filtration $(\mathcal{M}_n)_{n\pg1}$, we deduce that there exist $\alpha,\beta,\gamma\in D$ with the following properties,
\begin{enumerate}
    \item[{\rm(i)}] $y=\alpha+\beta+\gamma$,
    \item[{\rm(ii)}] the element $\alpha$ satisfies
\[\|\alpha\|_2^2\pp C^2\lambda\|y\|_1,\]
    \item[{\rm(iii)}] the element $\beta$ satisfies 
\[\sum_{n\pg1}\|\DD_n(\beta)\|_1\pp C\|y\|_1,\]
    \item[{\rm(iv)}] the element $\gamma$ satisfies
\[(1-p)\DD_n(\gamma)(1-p)=0\]
for every $n\pg1$ where $p\in\mathcal{M}$ is a projection such that
\[\tau(p)\pp C\frac{\|y\|_1}{\lambda},\]
\end{enumerate}
where $C>0$ is a universal constant. We set 
\[\alpha':=\PP(\alpha),\ \ \ \ \ \ \beta':=\PP(\beta),\ \ \ \ \ \ \ \gamma':=\PP(\gamma).\]
Clearly we have $\PP(y)=\alpha'+\beta'+\gamma'$. As $\PP$ is contractive on $D\cap L_2(\mathcal{M})$, the element $\alpha'$ satisfies
\[\|\alpha'\|_2^2\pp\|\alpha\|_2^2\pp C^2\lambda\|y\|_1.\]
The element $\beta'$ satisfies
\[\|\beta'\|_1=\Big\|\sum_{i\in I}\DD_i(\beta)\Big\|_1\pp\sum_{n\pg1}\|\DD_n(\beta)\|_1\pp C\|y\|_1.\]
Finally, the element $\gamma'$ satisfies
\[(1-p)\gamma'(1-p)=\sum_{i\in I}(1-p)\DD_i(\gamma)(1-p)=0.\]
The proof is complete.
\end{proof}

\begin{proof}[Proof of Fact \ref{Example 1 - Fact 2}]
It is clear from the definition that $\PP$ is weakly continuous on $(L_1+L_\infty)(\mathcal{M}_n)$ for every $n\pg1$. Moreover, $\PP$ clearly commutes with $\EE_n$, so $\EE_n$ stabilises $\PP(D)$ for every $n\pg1$. Knowing that $\PP$ is a Gundy projection, the fact is derived as a particular consequence of Corollary \ref{Gundy - Corollary 1}.
\end{proof}

\begin{fact}\label{Example 1 - Fact 3}
$(\PP_2(\mathcal{\mathcal{M}}),\PP_\infty(\mathcal{\mathcal{M}}))$ is $K$-closed in $(L_2(\mathcal{\mathcal{M}}),L_\infty(\mathcal{\mathcal{M}}))$ with a universal constant.
\end{fact}
\begin{proof}
The point is to use Pisier's duality lemma (Theorem \ref{PisierLemma}). It can be easily checked that $(\PP_1(\mathcal{M}),\PP_2(\mathcal{M}))$ is a regular couple. In addition, an easy computation using Proposition \ref{Example 1 - Proposition 1} shows that if $1\pp p,q\pp\infty$ are H\"older conjugate exponents, then
\[\PP_p(\mathcal{M})^{\bot}=\big\{x\in L_q(\mathcal{M})\ \mid\ \DD_i(x)=0,\ \forall i\in I\big\}\]
where $\PP_p(\mathcal{M})^{\bot}$ denotes the orthogonal of $\PP_p(\mathcal{M})$ with respect to the canonical duality between $L_p(\mathcal{M})$ and $L_q(\mathcal{M})$. As $(\PP_1(\mathcal{M}),\PP_2(\mathcal{M}))$ is $K$-closed in $(L_1(\mathcal{M}),L_2(\mathcal{M}))$, by Pisier's result, we deduce that $(\PP_2(\mathcal{M})^{\bot},\PP_1^{\bot}(\mathcal{M}))$ is $K$-closed in $(L_2(\mathcal{M},L_\infty(\mathcal{M}))$. We get the desired result by considering $\N^*\backslash I$ instead of $I$.
\end{proof}

\begin{fact}\label{Example 1 - Fact 4}
Let $1<p,q<\infty$. Then $(\PP_p(\mathcal{M}),\PP_q(\mathcal{M}))$ is complemented and thus $K$-closed in $(L_p(\mathcal{M}),L_q(\mathcal{M}))$ with a constant that depends only on $p,q$.
\end{fact}
\begin{proof}
From the boundedness properties of martingale transforms (see \cite{NarcisseMartingaleTransforms}), we know that $\PP$ is bounded for the $L_r$ norm with a constant that depends only on $r$, whenever $1<r<\infty$. This proves the point.
\end{proof}

\begin{fact}\label{Example 1 - Fact 5}
Let $1\pp p\neq q\pp\infty$, such that $(\PP_p(\mathcal{M}),\PP_q(\mathcal{M}))$ is $K$-closed in $(L_p(\mathcal{M}),L_q(\mathcal{M}))$. Then for every $0<\theta<1$, we have
\[(\PP_p(\mathcal{M}),\PP_q(\mathcal{M}))_{\theta,r}=\PP_r(\mathcal{M})\]
with equivalent norms, where $\frac{1}{r}=\frac{1-\theta}{p}+\frac{\theta}{q}$.
\end{fact}

The proof of this fact needs the following easy lemma.

\begin{lemm}\label{Example 1 - Lemma 2}
Let $1\pp p,q\pp\infty$ and $x\in\PP(D)\cap (L_p(\mathcal{M})+L_q(\mathcal{M}))$. Then, we have $x\in\PP_p(\mathcal{M})+\PP_q(\mathcal{M})$.
\end{lemm}
\begin{proof}[Proof of Lemma \ref{Example 1 - Lemma 2}]
As $x\in D$, there is $n\pg1$ such that $x=\sum_{1\pp k\pp n}\DD_k(x)$. Then, we have
\[x=\PP(x)=\PP\sum_{k=1}^{n}\DD_k(x)=\sum_{\protect\substack{k=1\\k\in I}}^{n}\DD_k(x).\]
Now, let $y\in L_p(\mathcal{M})$ and $z\in L_q(\mathcal{M})$ such that $x=y+z$. Then $x=y'+z'$ with
\[y':=\sum_{\protect\substack{k=1\\k\in I}}^{n}\DD_k(y)\in\PP_p(\mathcal{M}),\ \ \ \ \ \ \ \ \ \ z':=\sum_{\protect\substack{k=1\\k\in I}}^{n}\DD_k(z)\in\PP_q(\mathcal{M}).\]
The proof is complete.
\end{proof}

\begin{proof}[Proof of Fact \ref{Example 1 - Fact 5}]
Let $0<\theta<1$. Using Theorem \ref{FunctorLp} and Proposition \ref{Example 1 - Proposition 1} we clearly have a continuous inclusion
\[(\PP_p(\mathcal{M}),\PP_q(\mathcal{M}))_{\theta,r}\subset\PP_r(\mathcal{M})\]
where $1/r=(1-\theta)/p+\theta/q$, and by $K$-closedness it has a closed range. Thus, it suffices to show that it also has a dense range. Let $x\in\PP(D\cap L_r(\mathcal{M}))$. Then, it is clear that we have $x\in\PP(D)\cap (L_p(\mathcal{M})+L_q(\mathcal{M}))$, so by the previous lemma we have $x\in(\PP_p(\mathcal{M})+\PP_q(\mathcal{M}))\cap L_r(\mathcal{M})$. But by again using Theorem \ref{FunctorLp} and Proposition \ref{Kclosedness}, we know that \[(\PP_p(\mathcal{M})+\PP_q(\mathcal{M}))\cap L_r(\mathcal{M})=(\PP_p(\mathcal{M}),\PP_q(\mathcal{M}))_{\theta,r}.\]
Thus, $x\in(\PP_p(\mathcal{M}),\PP_q(\mathcal{M}))_{\theta,r}$. This concludes the proof because $\PP(D\cap L_r(\mathcal{M}))$ is norm-dense in $\PP_r(\mathcal{M})$ by definition (note that $r<\infty$).
\end{proof}

\subsubsection{Example 2}\label{Example2} Let $\mathcal{M}$ be a Neumann algebra equipped with a n.s.f. trace $\tau$ and also with two filtrations $(\mathcal{M}_n^{(0)})_{n\pg1},(\mathcal{M}_n^{(1)})_{n\pg1}$ with associated conditional expectations denoted by $(\EE_n^{(0)})_{n\pg1},(\EE_n^{(1)})_{n\pg1}$ and difference projections denoted by $(\DD_n^{(0)})_{n\pg1},(\DD_n^{(1)})_{n\pg1}$ respectively.  We assume that the two filtrations \textit{commute}, in the sense that
\begin{equation}\label{Example 2 - Assumption 1}
\EE_m^{(0)}\EE_n^{(1)}=\EE_n^{(1)}\EE_m^{(0)},
\end{equation}
for all $m,n\pg1$. As a consequence, we have
\[\DD_m^{(0)}\DD_n^{(1)}=\DD_n^{(1)}\DD_m^{(0)},\]
for all $m,n\pg1$. We also make the following additionnal assumption 
\begin{equation}\label{Example 2 - Assumption 2}
\mathcal{M}_n^{(0)}\subset\mathcal{M}_{\phi(n)}^{(1)},\ \ \ \ \ \ \ \ \ \ \ \ \ \mathcal{M}_n^{(1)}\subset\mathcal{M}_{\psi(n)}^{(0)},
\end{equation}
for all $n\pg1$ and for some maps $\phi,\psi$. As a particular consequence, we have
\[D:=\cup_{n\pg1}(L_1+L_\infty)(\mathcal{M}_n^{(0)})=\cup_{n\pg1}(L_1+L_\infty)(\mathcal{M}_n^{(1)}).\]
Now let $I^{(0)},I^{(1)}$ be two sets of positive integers such that
\begin{equation}\label{Example 2 - Assumption 3}
\DD_i^{(0)}\DD_j^{(1)}=\DD_j^{(1)}\DD_i^{(0)}=0,\ \ \ \ \ \ \ \text{for every}\ i\in I^{(0)},j\in I^{(1)}.
\end{equation}
Finally, let $\QQ:D\to D$ be the Lebesgue compatible linear idempotent operator given by the following expression,
\[\QQ(x)=\sum_{i\notin I^{(1)},j\notin I^{(1)}}\DD_i^{(0)}\DD_j^{(1)}(x)=\sum_{i\notin I^{(1)},j\notin I^{(1)}}\DD_j^{(1)}\DD_i^{(0)}(x),\ \ \ \ \ \ \ x\in D.\]
Notice that when the two filtrations coincide and one of the two sets $I^{(0)},I^{(1)}$ is empty, we recover the situation of Example 1.

\begin{prop}\label{Example 2 - Proposition 1}
Let $1\pp p\pp\infty$. Then
\[\QQ_p(\mathcal{M})=\big\{x\in L_p(\mathcal{M})\ \mid\ \DD_i^{(0)}(x)=0,\ \DD_j^{(1)}=0,\ \forall i\in I^{(0)},\ \forall j\in I^{(1)}\big\}.\]
\end{prop}
\begin{proof}
The proof follows the same lines as the proof of Proposition \ref{Example 1 - Proposition 1}. 
\end{proof}

Notice that $\QQ$ is the product of two commuting Gundy projections on $\mathcal{M}$ arising from Example 1. Namely, we have 
\[\QQ=\PP^{(0)}\PP^{(1)}=\PP^{(1)}\PP^{(0)}\] 
where for $j\in\{0,1\}$, the Gundy projection $\PP^{(j)}:D\to D$ is given by the following expression,
\[\PP^{(j)}(x)=\sum_{i\notin I^{(j)}}\DD_i^{(j)}(x),\ \ \ \ \ \ \ x\in D.\]
The fact that $\PP^{(0)}$ and $\PP^{(1)}$ commute follows directly from \eqref{Example 2 - Assumption 1}.
However, we will not be able to show that $\QQ$ is itself a Gundy projection.  This is why we have a restriction on the exponents in the following theorem.

\begin{theo}\label{Example 2 - Theorem 1}
Let $1<p,q\pp\infty$. Then the Banach couple $(\QQ_p(\mathcal{M}),\QQ_q(\mathcal{M}))$ is $K$-complemented in $(L_p(\mathcal{M}),L_q(\mathcal{M}))$ with a constant which depends only on $p,q$.
\end{theo}

As in Example 1, Theorem \ref{Example 2 - Theorem 1} follows from Facts 6-8 below.

\begin{fact}\label{Example 2 - Fact 1}
For every $1\pp p,q\pp\infty$, we have $(\QQ_p(\mathcal{M})+\QQ_q(\mathcal{M}))\cap L_p(\mathcal{M})=\QQ_p(\mathcal{M})$.
\end{fact}
\begin{proof}
Again, this follows trivially from the expression given in Proposition \ref{Example 2 - Proposition 1}.
\end{proof}

\begin{fact}\label{Example 2 - Fact 2}
$(\QQ_2(\mathcal{M}),\QQ_\infty(\mathcal{M}))$ is $K$-closed in $(L_2(\mathcal{M}),L_\infty(\mathcal{M}))$ with a universal constant.
\end{fact}

Now, we turn to the proof of Fact \ref{Example 2 - Fact 2}. Let $\I$ denote the identity operator on $D$ and let
\[\PP:=\II-\QQ\]
denote the complement of $\QQ$. Then, for $j\in\{0,1\}$ we have
\[(\II-\PP^{(j)})(x)=\sum_{i\in I^{(j)}}\DD_i^{(j)}(x),\ \ \ \ \ \ \ x\in D.\]
We know that $\II-\PP^{(j)}$ is a Gundy projection as a particular case of the previous example.
In addition, from \eqref{Example 2 - Assumption 3}, we deduce that the projections $\II-\PP^{(0)}$ and $\II-\PP^{(1)}$ are orthogonal. Using \ref{Gundy - Proposition 1}, we deduce that $(\II-\PP^{(0)})+(\II-\PP^{(1)})$ is a Gundy projection. But little algebra yields the identity
\[\PP=(\II-\PP^{(0)})+(\II-\PP^{(1)}).\]
Hence, $\PP$ is a Gundy projection. Besides,
from \eqref{Example 2 - Assumption 1}, it follows that $\PP$ commutes with $\EE_n^{(0)}$, for every $n\pg1$. Moreover, from \eqref{Example 2 - Assumption 2}, it follows that
$\PP$ is weakly continuous on $(L_1+L_\infty)(\mathcal{M}_n^{(0)})$, for every $n\pg1$. Knowing that $\PP$ is a Gundy projection, using Corollary \ref{Gundy - Corollary 1}, we deduce the following lemma.

\begin{lemm}\label{Example 2 - Lemma 3}
$(\PP_1(\mathcal{M}),\PP_2(\mathcal{M}))$ is $K$-closed in the couple $(L_1(\mathcal{M}),L_2(\mathcal{M}))$ with a universal constant.
\end{lemm}

\begin{proof}[Proof of Fact \ref{Example 2 - Fact 2}]
Again, the point is to use Pisier's duality lemma. First, we easily check that $(\PP_1(\mathcal{M}),\PP_2(\mathcal{M}))$ is a regular couple. In addition, an easy computation shows that if $1\pp p,q\pp\infty$ are H\"older conjugate exponents, then
\[\PP_p(\mathcal{M})^{\bot}=\QQ_q(\mathcal{M})\]
where $\PP_p(\mathcal{M})^{\bot}$ denotes the orthogonal of $\PP_p(\mathcal{M})$ with respect to the canonical duality between $L_p(\mathcal{M})$ and $L_q(\mathcal{M})$. As $(\PP_1(\mathcal{M}),\PP_2(\mathcal{M}))$ is $K$-closed in $(L_1(\mathcal{M}),L_2(\mathcal{M}))$, by Pisier's result we deduce that $(\QQ_2(\mathcal{M}),\QQ_\infty(\mathcal{M}))$ is $K$-closed in $(L_2(\mathcal{M}),L_\infty(\mathcal{M}))$. We omit the details.
\end{proof}

\begin{fact}\label{Example 2 - Fact 3}
Let $1<p,q<\infty$. Then $(\QQ_p(\mathcal{M}),\QQ_q(\mathcal{M}))$ is complemented and thus $K$-closed in $(L_p(\mathcal{M}),L_q(\mathcal{M}))$ with a constant that depends only on $p,q$.
\end{fact}
\begin{proof}
From the boundedness properties of martingale transforms, we know that $\PP^{(0)}$ and $\PP^{(1)}$, and thus, also $\QQ=\PP^{(0)}\PP^{(1)}=\PP^{(1)}\PP^{(0)}$, are bounded for the $L_r$-norm with a constant that depends only on $r$, whenever $1<r<\infty$. This proves the point.
\end{proof}

\begin{fact}\label{Example 2 - Fact 4}
Let $1<p<q\pp\infty$, such that $(\QQ_p(\mathcal{M}),\QQ_q(\mathcal{M}))$ is $K$-closed in the couple $(L_p(\mathcal{M}),L_q(\mathcal{M}))$. Then for every $0<\theta<1$, we have
\[(\QQ_p(\mathcal{M}),\QQ_q(\mathcal{M}))_{\theta,r}=\QQ_r(\mathcal{M})\]
with equivalent norms, where $\frac{1}{r}=\frac{1-\theta}{p}+\frac{\theta}{q}$.
\end{fact}
\begin{proof}
The proof follows the same lines as the proof of Fact \ref{Example 1 - Fact 5}. 
\end{proof}

\subsubsection{Example 3}\label{Example3}
Let $\mathcal{M},\mathcal{N}$ be two von Neumann algebras each equipped respectively with a n.s.f. trace $\tau$, $\sigma$ and a filtration $(\mathcal{M}_n)_{n\pg1},(\mathcal{N}_n)_{n\pg1}$, and with associated conditional expectations respectively denoted by $(\EE_n)_{n\pg1},(\FF_n)_{n\pg1}$. We consider the von Neumann tensor product algebra
\[\mathcal{T}:=\mathcal{M}\bar{\otimes}\mathcal{N}\]
equipped with the tensor product trace. For every $n\pg1$, we set
\[\mathcal{T}_{2n-1}^{(0)}:=\mathcal{M}_n\bar{\otimes}\mathcal{N}_n,\ \ \ \ \ \ \ \ \ \ \ \ \ \mathcal{T}_{2n}^{(0)}:=\mathcal{M}_{n+1}\bar{\otimes}\mathcal{N}_n,\]
and
\[\mathcal{T}_{2n-1}^{(1)}:=\mathcal{M}_n\bar{\otimes}\mathcal{N}_{n+1},\ \ \ \ \ \ \ \ \ \ \ \ \ \mathcal{T}_{2n}^{(1)}:=\mathcal{M}_{n+1}\bar{\otimes}\mathcal{N}_{n+1}.\]
Then $(\mathcal{T}_n^{(0)})_{n\pg1}$, $(\mathcal{T}_n^{(1)})_{n\pg1}$ are two filtrations on $\mathcal{T}$. 
Let $(\EE_n^{(0)})_{n\pg1},(\EE_n^{(1)})_{n\pg1}$ denote their associated conditional expectations and let $(\DD_n^{(0)})_{n\pg1}$, $(\DD_n^{(1)})_{n\pg1}$ denote their associated difference projections. 

\begin{lemm}
We have
\begin{enumerate}
    \item for all $m,n\pg1$,
\[\EE_m^{(0)}\EE_n^{(1)}=\EE_n^{(1)}\EE_m^{(0)},\]
    \item for all $n\pg1$,
    \[\mathcal{T}_n^{(0)}\subset\mathcal{T}_{n}^{(1)},\ \ \ \ \ \ \ \ \ \ \ \ \ \mathcal{T}_n^{(1)}\subset\mathcal{T}_{n+2}^{(0)},\]
    \item and for all $m,n\pg1$,
\[\DD_{2m}^{(0)}\DD_{2n-1}^{(1)}=\DD_{2n-1}^{(1)}\DD_{2m}^{(0)}=0.\]
\end{enumerate}
\end{lemm}
\begin{proof}
The first two assertions are obvious. For the last one, we fix $m\pg1$. If $n\pg2$, then by easy computation we obtain
\[\DD_{2m}^{(0)}\DD_{2n-1}^{(1)}=\EE_n(\EE_{m+1}-\EE_m)\otimes\FF_m(\FF_{n+1}-\FF_{n})\]
If $m\pp n$, then $\FF_m(\FF_{n+1}-\FF_{n})=0$ and if $m\pg n$, then $\EE_n(\EE_{m+1}-\EE_m)=0$. Thus, in all cases, the above is zero. If $n=1$, then $\DD_{2m}^{(0)}\DD_{1}^{(1)}=\EE_1(\EE_{m+1}-\EE_m)\otimes\FF_m=0$.
\end{proof}

The above lemma shows that with the choice
\[I^{(0)}:=\{2m\mid n\pg1\},\ \ \ \ \ \ I^{(1)}:=\{2n-1\mid n\pg1\},\]
we are exactly in the situation of Example 2. Let $\QQ:D\to D$ be the associated Lebesgue compatible linear idempotent operator as defined there. Recall that
\[D:=\cup_{n\pg1}(L_1+L_\infty)(\mathcal{T}_n^{(0)})=\cup_{n\pg1}(L_1+L_\infty)(\mathcal{T}_n^{(1)}).\]
and
\[\QQ(x)=\sum_{m,n\pg1}\DD_{2m-1}^{(0)}\DD_{2n}^{(1)}(x)=\sum_{m,n\pg1}\DD_{2n}^{(1)}\DD_{2m-1}^{(0)}(x),\ \ \ \ \ \ \ x\in D.\]
Moreover, for every $1\pp p\pp\infty$, we have
\[\QQ_p(\mathcal{T}):=\big\{y\in L_p(\mathcal{T})\ \mid\ \forall n\pg1,\ \DD_{2n}^{(0)}(y)=\DD_{2n-1}^{(1)}(y)=0\big\}.\]
The results of the previous example directly produce the following four facts.

\begin{fact}\label{Example 3 - Fact 1}
For every $1\pp p,q\pp\infty$, we have $(\QQ_p(\mathcal{T})+\QQ_q(\mathcal{T}))\cap L_p(\mathcal{T})=\QQ_p(\mathcal{T})$.
\end{fact}

\begin{fact}\label{Example 3 - Fact 2}
Let $1<p<q\pp\infty$, such that $(\QQ_p(\mathcal{T}),\QQ_q(\mathcal{T}))$ is $K$-closed in $(L_p(\mathcal{T}),L_q(\mathcal{T}))$. Then for every $0<\theta<1$, we have
\[(\QQ_p(\mathcal{T}),\QQ_q(\mathcal{T}))_{\theta,r}=\QQ_r(\mathcal{M})\]
with equivalent norms, where $\frac{1}{r}=\frac{1-\theta}{p}+\frac{\theta}{q}$.
\end{fact}

\begin{fact}\label{Example 3 - Fact 3}
$(\QQ_2(\mathcal{T}),\QQ_\infty(\mathcal{T}))$ is $K$-closed in $(L_2(\mathcal{T}),L_\infty(\mathcal{T}))$ with a universal constant.
\end{fact}

\begin{fact}\label{Example 3 - Fact 4}
Let $1<p,q<\infty$. Then, $(\QQ_p(\mathcal{T}),\QQ_q(\mathcal{T}))$ is complemented and thus $K$-closed in $(L_p(\mathcal{T}),L_q(\mathcal{T}))$ with a constant that depends only on $p,q$.
\end{fact}

To the above list of facts, we add the following new one.

\begin{fact}\label{Example 3 - Fact 5}
$(\QQ_1(\mathcal{T}),\QQ_2(\mathcal{T}))$ is $K$-closed in $(L_1(\mathcal{T}),L_2(\mathcal{T}))$ with a universal constant.
\end{fact}

By combining Facts 10-14 above with the usual interpolation machinery, we obtain the following theorem.

\begin{theo}\label{Example 3 - Theorem 1}
Let $1\pp p,q\pp\infty$. Then the Banach couple $(\QQ_p(\mathcal{T}),\QQ_q(\mathcal{T}))$ is $K$-complemented in $(L_p(\mathcal{T}),L_q(\mathcal{T}))$ with a universal constant.
\end{theo}

Now, it remains to justify Fact \ref{Example 3 - Fact 5}. 
The key point is to introduce another filtration on $\mathcal{T}$ and then to use carefully Gundy's decomposition theorem with respect to this filtration to show that $\QQ$ is a Gundy projection. That is, we consider the filtration $(\mathcal{T}_n)_{n\pg1}$ such that 
\[\mathcal{T}_n:=\mathcal{M}_n\otimes\mathcal{N}_n,\ \ \ \ \ \ n\pg1.\] 
We denote by $(\DD_n)_{n\pg1}$ the associated difference projections. Thus, we have
\[\DD_m=\EE_m\otimes\FF_m-\EE_{m-1}\otimes\FF_{m-1}\ \ \ \ \ \ \text{for every}\ m\pg1.\]
(with the convention $\EE_0=\FF_0=0$.)

\begin{lemm}\label{Example 3 - Lemma 1}
We have
\[\QQ(x)=\sum_{n\pg1}\TT_n(x),\ \ \ \ \ \ \ \text{for}\ x\in D,\]
where we set $\TT_n:=(\EE_{n+1}-\EE_n)\otimes(\FF_{n+1}-\FF_n)$ for every $n\pg1$.
\end{lemm}
\begin{proof}
If $n,m\pg1$, an easy computation yields
\[\DD^{(0)}_{2m-1}\DD^{(1)}_{2n}=\EE_m(\EE_{n+1}-\EE_n)\otimes(\FF_m-\FF_{m-1})\FF_{n+1}\]
But if $m\pp n$, then $\EE_m(\EE_{n+1}-\EE_n)=0$, and if $m\pg n+2$, then $(\FF_m-\FF_{m-1})\FF_{n+1}=0$. Thus, $\DD^{(0)}_{2m-1}\DD^{(1)}_{2n}$ is always $0$ except in the case $m=n+1$, in which case it is equal to $\TT_n$. This completes the proof of the lemma.
\end{proof}

\begin{lemm}\label{Example 3 - Lemma 2}
We have
\[\QQ(x)=\sum_{n\pg1}\TT_n(\DD_{n+1}(x)),\ \ \ \ \ \ \text{for}\ x\in D.\]
\end{lemm}
\begin{proof}
If $n,m\pg1$, we have
\begin{align*}
\TT_n\DD_{m}&=((\EE_{n+1}-\EE_n)\otimes(\FF_{n+1}-\FF_n))(\EE_m\otimes\FF_m-\EE_{m-1}\otimes\FF_{m-1})\\
&=\EE_m(\EE_{n+1}-\EE_n)\otimes\FF_m(\FF_{n+1}-\FF_n)-\EE_{m-1}(\EE_{n+1}-\EE_n)\otimes\FF_{m-1}(\FF_{n+1}-\FF_n)
\end{align*}
The above computation shows that $\TT_n\DD_m$ is always $0$ except in the case $m=n+1$. Thus, if $x\in D$, we have
\[\QQ(x)=\sum_{m\pg1}\QQ(\DD_m(x))=\sum_{m\pg1}\sum_{n\pg1}\TT_n(\DD_m(x))=\sum_{n\pg1}\TT_n(\DD_{n+1}(x)).\]
\end{proof}

\begin{lemm}\label{Example 3 - Lemma 3}
$\QQ$ is a Gundy projection on $\mathcal{T}$ with a universal constant.
\end{lemm}
\begin{proof}
It is clear that $\QQ$ is self-adjoint on $L_2(\mathcal{T})$. Let $y\in D\cap L_1(\mathcal{T})$ and $\lambda>0$. First, we assume that $y$ is positive. Then, by applying Theorem \ref{ParcetRandria} to $y$ and $\lambda$ and with respect to the filtration $(\mathcal{T}_n)_{n\pg1}$, we deduce that there exist $\alpha,\beta,\gamma\in D$ and a sequence of projections $(q_n)_{n\pg1}$ of $\mathcal{T}$, with $q_n\in\mathcal{T}_n$ for $n\pg1$, such that
\begin{enumerate}
    \item[{\rm(i)}] $y=\alpha+\beta+\gamma$,
    \item[{\rm(ii)}] the element $\alpha$ satisfies
\[\|\alpha\|_2^2\pp C^2\lambda\|y\|_1,\]
    \item[{\rm(iii)}] the element $\beta$ satisfies 
\[\sum_{n\pg1}\|\DD_n(\beta)\|_1\pp C\|y\|_1,\]
    
    \item[{\rm(iv)}] the element $\gamma$ satisfies
\[\DD_{n+1}(\gamma)=\DD_{n+1}(y)-q_n\DD_{n+1}(y)q_n\]
for every $n\pg1$, 
    \item[{\rm(v)}] the projection $p:=1-\wedge_{n\pg1}q_n$ satisfies
\[\tau(p)\pp C\frac{\|y\|_1}{\lambda},\]
\end{enumerate} 
where $C>0$ is a universal constant. We set 
\[\alpha':=\QQ(\alpha),\ \ \ \ \ \ \beta':=\QQ(\beta),\ \ \ \ \ \ \ \gamma':=\QQ(\gamma).\]
Clearly we have $\QQ(y)=\alpha'+\beta'+\gamma'$. As $\QQ$ is contractive on $D\cap L_2(\mathcal{M})$, the element $\alpha'$ satisfies
\[\|\alpha'\|_2^2\pp\|\alpha\|_2^2\pp C^2\lambda\|y\|_1.\]
By using the expression of $\QQ$ given in Lemma \ref{Example 3 - Lemma 2}, the element $\beta'$ satisfies
\[\|\beta'\|_1=\|\QQ(\beta)\|_1=\Big\|\sum_{n\pg1}\TT_n(\DD_{n+1}(\beta))\Big\|_1\pp2\sum_{n\pg1}\|\DD_{n+1}(\beta)\|_1\pp2C\|y\|_1.\]
Now fix $n\pg1$. Since $\EE_{n+1}-\EE_n$ is $\mathcal{M}_n$-linear and $\FF_{n+1}-\FF_n$ is $\mathcal{N}_n$-linear, we deduce that $\TT_n=(\EE_{n+1}-\EE_n)\otimes(\FF_{n+1}-\FF_n)$ is 
$\mathcal{T}_n=\mathcal{M}_n\otimes\mathcal{N}_n
$-linear, and therefore, as $q_n\in\mathcal{T}_n$, we have
\begin{align*}
q_n\TT_n(\DD_{n+1}(\gamma))q_n&=\TT_n(q_n\DD_{n+1}(\gamma)q_n)\\
&=\TT_n(q_n(\DD_{n+1}(y)-q_n\DD_{n+1}(y)q_n)q_n)\\
&=\TT_n(q_n\DD_{n+1}(y)q_n-q_n\DD_{n+1}(y)q_n))\\
&=0,
\end{align*}
and it direcly follows that
\[(1-p)\TT_n(\DD_{n+1}(\gamma))(1-p)=0.\]
From the expression of $\QQ$ obtained in Lemma \ref{Example 3 - Lemma 2}, we deduce that the element $\gamma'$ satisfies
\[(1-p)\gamma'(1-p)=(1-p)\QQ(\gamma)(1-p)=\sum_{n\pg1}(1-p)\TT_n(\DD_{n+1}(\gamma))(1-p)=0,\]
as required. In the general case where $y$ is no longer assumed to be positive, one can decompose $y=(y_1^{+}-y_1^{-})+i(y_2^{+}-y_2^{-})$ where $y_j^{\pm}\in D\cap L_1(\mathcal{T})$ is positive and satisfies $\|y_j^{\pm}\|_1\pp\|y\|_1$. Then by applying what we just proved for each of the $y_j^{\pm}$, one gets the desired conclusion.
\end{proof}

To complete the proof of Fact \ref{Example 3 - Fact 5}, it suffices to observe that $\QQ$ satisfies the hypothesis of Corollary \ref{Gundy - Corollary 1} with respect to either of the three filtrations involved. We omit the details.

\section{Applications}

In this last section, we use the material from the previous one to derive new interpolation results for noncommutative adapted spaces and noncommutative Hardy spaces. 

The section is organised as follows. In the first paragraph, we provide some background on column/row/mixed spaces that will be needed subsequently. In the second paragraph, we establish a $K$-closedness result for noncommutative row/column/mixed adapted spaces by using the results contained in Example 1 of Section 2. In the third paragraph, we establish an analogous $K$-closedness result for noncommutative column/row/mixed Hardy spaces by using the results contained in Example 3 of Section 2. 

\subsection{Preliminaries: Column/Row/Mixed spaces} Column/Row/Mixed spaces are particular instances of Hilbert space-valued Lebesgue spaces. We refer to \cite{JungeSquareFunctions} for the subject. Let $\mathcal{M}$ be a von Neumann algebra equipped with a n.s.f. trace. 

For every $1\pp p\pp\infty$, the \textit{column space} $L_p(\mathcal{M},\ell_2^c)$ (resp. the \textit{row space} $L_p(\mathcal{M},\ell_2^r)$) is defined as the set of sequences $x=(x_n)_{n\pg1}$ of $L_p(\mathcal{M})$ such that the increasing sequence
\[\Big(\Big(\sum_{n=1}^{N}x_n^*x_n\Big)^{1/2}\Big)_{N\pg1}\ \ \ \text{resp.}\ \ \Big(\Big(\sum_{n=1}^{N}x_nx_n^*\Big)^{1/2}\Big)_{N\pg1}\]
is norm bounded in $L_p(\mathcal{M})$. Then $L_p(\mathcal{M},\ell_2^c)$ (resp. $L_p(\mathcal{M},\ell_2^r)$) is a Banach space for the norm $\|\cdot\|_{L_p(\mathcal{M},\ell_2^c)}$ (resp. $\|\cdot\|_{L_p(\mathcal{M},\ell_2^r)}$) defined as follows,
\[\|x\|_{L_p(\mathcal{M},\ell_2^c)}:=\sup_{N\pg1}\Big\|\Big(\sum_{n=1}^{N}x_n^*x_n\Big)^{1/2}\Big\|_{L_p(\mathcal{M})}\]
resp.
\[\|x\|_{L_p(\mathcal{M},\ell_2^r)}:=\sup_{N\pg1}\Big\|\Big(\sum_{n=1}^{N}x_nx_n^*\Big)^{1/2}\Big\|_{L_p(\mathcal{M})}.\]
For every $1\pp p\pp\infty$, the \textit{mixed space} $L_p(\mathcal{M},\ell_2^{rc})$ is defined as follows,
\[L_p(\mathcal{M},\ell_2^{rc}):=\left\{\begin{array}{cl}
    L_p(\mathcal{M},\ell_2^r)\cap L_p(\mathcal{M},\ell_2^c) & {\rm if}\ 2\pp p\pp\infty \\
    L_p(\mathcal{M},\ell_2^r)+L_p(\mathcal{M},\ell_2^c) & {\rm if}\ 1\pp p<2 
\end{array}\right..\]
Then $L_p(\mathcal{M},\ell_2^{rc})$ is a Banach space for the norm $\|\cdot\|_{L_p(\mathcal{M},\ell_2^{rc})}$ defined as follows,
\[\|x\|_{L_p(\mathcal{M},\ell_2^{rc})}:=\left\{\begin{array}{cl}
    \max\big\{\|x\|_{L_p(\mathcal{M},\ell_2^r)},\|x\|_{L_p(\mathcal{M},\ell_2^c)}\big\} & {\rm if}\ 2\pp p\pp\infty \\
    \displaystyle\inf_{x=a+b}\|a\|_{L_p(\mathcal{M},\ell_2^r)}+\|b\|_{L_p(\mathcal{M},\ell_2^c)}& {\rm if}\ 1\pp p<2 
\end{array}\right..\]

Many features of the class of noncommutative Lebesgue spaces, including H\"older's inequality and trace duality, have an analogous version for the three classes of column/row/mixed spaces. \textbf{In the sequel, when the symbols $r,c,rc$ are omitted, this means that the statement holds equally for the three settings of column/row/mixed spaces.}

First, notice that if $1\pp p,q\pp\infty$, the inclusion $L_p(\mathcal{M},\ell_2)\subset L_p(\mathcal{M})^{\N^*}$ is continuous with respect to the product topology, so that, if $1\pp p,q\pp\infty$, the sum space
\[(L_p+L_q)(\mathcal{M},\ell_2):=L_p(\mathcal{M},\ell_2)+L_q(\mathcal{M},\ell_2)\]
becomes a Banach space for the sum norm. An important fact is that the space of eventually-zero sequences of $(L_p+L_q)(\mathcal{M})$ is weakly dense in $(L_p+L_q)(\mathcal{M},\ell_2)$. More precisely, if $x=(x_n)_{n\pg1}\in(L_p+L_q)(\mathcal{M},\ell_2)$, and if we set $x^m:=(1_{n\pp m}x_n)_{n\pg1}\in(L_p+L_q)(\mathcal{M},\ell_2)$ for every $m\pg1$, then the sequence $(x^m)_{m\pg1}$ converges weakly to $x$. The most important feature of column/row/mixed spaces is that they nicely embed as closed subspaces of a noncommutative Lebesgue space with the same exponent on a suitable von Neumann algebra equipped with a n.s.f. trace. This yields the following crucial lemma, which will be used several times in the sequel. 

\begin{lemm}\label{AdaptedLemma0}
There is,
\begin{enumerate}
    \item[{\rm a)}] an auxiliary von Neumann algebra $\mathcal{N}$ endowed with a n.s.f. trace,
    \item[{\rm b)}] a weak*-dense *-subalgebra $\mathcal{A}$ of $\mathcal{N}$.
    \item[{\rm c)}] a linear base $X$ for $\mathcal{A}$,
    \item[{\rm d)}] and a sequence $(\xi_n)_{n\pg1}$ of $X$ that generates $\mathcal{A}$ as an *-subalgebra,
\end{enumerate}
such that, if $\mathcal{T}:=\mathcal{M}\bar{\otimes}\mathcal{N}$ denotes the tensor product von Neumann algebra equipped with the tensor product trace, 
\begin{enumerate} 
    \item[{\rm0)}] the $(\xi_n)_{n\pg1}$ belong to $(L_1\cap L_\infty)(\mathcal{N})$, are orthonormal in $L_2(\mathcal{N})$, and for every $n\pg1$ the trace of $\mathcal{N}$ is finite on the von Neumann subalgebra generated by $(\xi_n)_{1\pp k\pp n}$,
    \item[{\rm1)}] the mapping $T(x_n)_{n\pg1}=\sum_{n\pg1}x_n\otimes\xi_n$ well defined on the subspace of eventually-zero sequences of $(L_1+L_\infty)(\mathcal{M})$, extends to
    a weakly continuous injective operator
    \[T:(L_1+L_\infty)(\mathcal{M},\ell_2)\to(L_1+L_\infty)(\mathcal{T})\]
    such that whenever $x\in(L_1+L_\infty)(\mathcal{M},\ell_2)$ and $1\pp p\pp\infty$, we have $x\in L_p(\mathcal{T})$ if and only if $Tx\in L_p(\mathcal{T})$. In that case
\[\|x\|_{L_p(\mathcal{M},\ell_2)}\pp\|Tx\|_{L_p(\mathcal{T})}\pp 2\|x\|_{L_p(\mathcal{M},\ell_2)},\]
    \item[{\rm2)}] the mapping $P(\sum_{\xi\in X}x_\xi\otimes\xi)=\sum_{n\pg1}x_{\xi_n}\otimes\xi_n$ well defined on the algebraic tensor product $(L_1+L_\infty)(\mathcal{M})\otimes\mathcal{A}$ extends to a weakly continuous idempotent operator 
\[P:(L_1+L_\infty)(\mathcal{T})\to(L_1+L_\infty)(\mathcal{T})\]
such that whenever $1\pp p\pp\infty$ and $y\in L_p(\mathcal{T})$ we have $Py\in L_p(\mathcal{T})$ with 
\[\|Py\|_{L_p(\mathcal{T})}\pp2\|y\|_{L_p(\mathcal{T})},\]
\end{enumerate}
\end{lemm}
\begin{proof}
We need to deal with row/column spaces and mixed spaces separately.

\textit{For row/column spaces.} Let $\ell_2(\N^*)$ denote the Hilbert space of square-summable complex sequences. Let $\mathcal{N}$ be the von Neumann algebra of bounded operators on $\ell_2(\N^*)$ endowed with its canonical trace. Let $\mathcal{A}$ be the set of bounded operators on $\ell_2(\N^*)$ with finite support. Then $\mathcal{A}$ is clearly a weak*-dense *-subalgebra of $\mathcal{N}$. Let $X:=\{e_{ij}\ \mid\ i,j\pg1\}$ be the set of canonical elementary operators on $\ell_2$. Then clearly $X$ is a linear base for $\mathcal{A}$. Finally, let $(\xi_n)_{n\pg1}$ denote the sequence of $X$ such that $\xi_n=e_{1n}$ in the setting of row spaces and $\xi=e_{n1}$ in the setting of column spaces, for every $n\pg1$. Then $X$ clearly generates $\mathcal{A}$ as a *-subalgebra. Assertions 0), 1) and 2) are easily checked. Notice that the idempotent operator $P:(L_1+L_\infty)(\mathcal{T})\to(L_1+L_\infty)(\mathcal{T})$ acts by multiplication by the projection $1\otimes e_{11}\in L_\infty(\mathcal{T})$, on the left in the row spaces setting and on the right in the column spaces setting.

\textit{For mixed spaces.} Let $\mathbb{F}_\infty$ denote the free group with generators $(g_n)_{n\pg1}$. Let $\mathcal{N}$ be the group von Neumann algebra of $\mathbb{F}_\infty$ endowed with its canonical trace. Let $\mathcal{A}$ be the group algebra of $\mathbb{F}_\infty$, viewed as a subset of $\mathcal{N}$. Then $\mathcal{A}$ is clearly a weak*-dense *-subalgebra of $\mathcal{N}$ included in $(L_1\cap L_\infty)(\mathcal{N})$. Let $X:=\{\lambda_g\ \mid\ g\in \mathbb{F}_\infty\}$ be the range of the regular representation of $\mathbb{F}_\infty$. Then of course $X$ is a linear base for $\mathcal{A}$. Finally, let $(\xi_n)_{n\pg1}$ denote the sequence of $X$ such that $\xi_n=\lambda_{g_n}$ for every $n\pg1$. Then $X$ clearly generates $\mathcal{A}$ as a *-subalgebra. Assertion 0) is easily checked, while assertions 1) and 2) follow from Khintchine's inequality (see \cite{HaagerupKhintchine, RicardKhintchine}). We omit the details.
\end{proof}

\begin{rem}
The injective operator $T:(L_1+L_\infty)(\mathcal{M},\ell_2)\to(L_1+L_\infty)(\mathcal{T})$ and the idempotent operator $P:(L_1+L_\infty)(\mathcal{T})\to(L_1+L_\infty)(\mathcal{T})$ have the same range, namely the closure in $(L_1+L_\infty)(\mathcal{T})$ of the set
\[\{x\otimes\xi_n\ \mid\ x\in(L_1+L_\infty)(\mathcal{M}),\ n\pg1\}.\]    
\end{rem}

We now investigate the properties of the Banach couple $(L_1(\mathcal{M},\ell_2),L_\infty(\mathcal{M},\ell_2))$. Using Theorem \ref{InterpLp}, one can easily deduce the following result.

\begin{theo}\label{InterpRC}
For every $0<\theta<1$, we have
\[(L_1(\mathcal{M},\ell_2),L_\infty(\mathcal{M},\ell_2))_{\theta,p}=L_p(\mathcal{M},\ell_2)\]
with equivalent norms, and constants independant of $\mathcal{M}$, where $1/p=1-\theta$. 
\end{theo}   

If $E$ is a fully symmetric intermediate space (or an exact interpolation space) for $(L_1,L_\infty)$, with the notation of Lemma \ref{AdaptedLemma0}, then the set
\[E(\mathcal{M},\ell_2):=\left\{x\in (L_1+L_\infty)(\mathcal{M},\ell_2)\ \mid\ Tx\in E(\mathcal{T})\right\}\]
becomes a Banach space when equipped with the norm $\|\cdot\|_{E(\mathcal{M},\ell_2^{rc})}$ defined as follows,
\[\|x\|_{E(\mathcal{M},\ell_2^{rc})}:=\|Tx\|_{E(\mathcal{T})}.\]
Furthermore, $E(\mathcal{M},\ell_2)$ is clearly an interpolation space for $(L_1(\mathcal{M},\ell_2),L_\infty(\mathcal{M},\ell_2))$. More generally, from Proposition \ref{Dodds2}, we deduce the following result.

\begin{prop}\label{Dodds3}
Let $\mathcal{M},\mathcal{N}$ be von Neumann algebras equipped with n.s.f. traces. Then for every fully symmetric intermediate space $E$ for $(L_1,L_\infty)$, the pair $E(\mathcal{M},\ell_2)$, $E(\mathcal{N},\ell_2)$ is an interpolation pair for $(L_1(\mathcal{M},\ell_2),L_\infty(\mathcal{M},\ell_2))$, $(L_1(\mathcal{N},\ell_2),L_\infty(\mathcal{N},\ell_2))$. 
\end{prop}

By the above result, one can deduce that if $x=(x_n)_{n\pg1}$ is a sequence $(L_1+L_\infty)(\mathcal{M})$ that belongs to $E(\mathcal{M},\ell_2)$, then $x_n\in E(\mathcal{M})$ for every $n\pg1$. As a consequence, with the choice $E:=L_p$ with $1\pp p\pp\infty$, we have $E(\mathcal{M},\ell_2)=L_p(\mathcal{M},\ell_2)$ with equivalent norms, with universal constants. Another important fact is that every eventually-zero sequence of $E(\mathcal{M})$ belongs to $E(\mathcal{M},\ell_2)$. However, in general, eventually-zero sequences of $E(\mathcal{M})$ may not be dense in $E(\mathcal{M},\ell_2)$. For example, it is not the case when $E:=L_\infty$. We say that $E$ \textit{satisfies the separability condition} \hypertarget{separability}{(S)} if for every $x=(x_n)_{n\pg1}\in E(\mathcal{M},\ell_2)$, if we set $x^m:=(1_{n\pp m}x_n)_{n\pg1}\in E(\mathcal{M},\ell_2)$ for $m\pg1$, then the sequence $(x^m)_{m\pg1}$ converges to $x$ in $E(\mathcal{M},\ell_2)$. Note that for every $1\pp p,q<\infty$, $E:=L_p+L_q$ satisfies the separability condition.


The following result, which is easily deduced from Theorem \ref{FunctorLp}, provides a generalisation of Theorem \ref{InterpRC}.

\begin{theo}\label{FunctorRC}
Let $\mathcal{F}$ be an exact interpolation functor and consider the exact interpolation space $E:=\mathcal{F}(L_1,L_\infty)$. Then
\[E(\mathcal{M},\ell_2)=\mathcal{F}(L_1(\mathcal{M},\ell_2),L_\infty(\mathcal{M},\ell_2))\]
with equivalent norms and constants independent of $\mathcal{M}$.
\end{theo}

\subsection{Warm-up: Noncommutative adapted spaces}

Let $\mathcal{M}$ be a von Neumann algebra endowed with a n.s.f. trace and a filtration denoted $(\mathcal{M}_n)_{n\pg1}$ with associated conditional expectations denoted $(\EE_n)_{n\pg1}$. Recall that a sequence $x=(x_n)_{n\pg1}$ of $(L_1+L_\infty)(\mathcal{M})$ is said to be adapted if $\EE_n(x_n)=x_n$ for every $n\pg1$. For every $1\pp p\pp\infty$, we consider the \textit{adapted column/row spaces},
\[L_p^{{\rm ad}}(\mathcal{M},\ell_2^c):=\big\{x\in L_p(\mathcal{M},\ell_2^c)\ \mid\ x{\rm\ is\ an\ adapted\ sequence}\big\},\]
\[L_p^{{\rm ad}}(\mathcal{M},\ell_2^r):=\big\{x\in L_p(\mathcal{M},\ell_2^r)\ \mid\ x{\rm\ is\ an\ adapted\ sequence}\big\},\]
and the \textit{adapted mixed space},
\[L_p^{{\rm ad}}(\mathcal{M},\ell_2^{rc}):=\big\{x\in L_p(\mathcal{M},\ell_2^{rc})\ \mid\ x{\rm\ is\ an\ adapted\ sequence}\big\}.\]
It is clear that these adapted row/column/mixed spaces are closed in the corresponding row/column/mixed spaces. The main result of this paragraph is the following one, and as said in the introduction, it can be considered as an extension of the noncommutative version of Stein's inequality.

\begin{theo}\label{AdaptedTheorem}
Let $1\pp p,q\pp\infty$. Then, $(L_p^{{\rm ad}}(\mathcal{M},\ell_2),L_q^{{\rm ad}}(\mathcal{M},\ell_2))$ is $K$-complemented in $(L_p(\mathcal{M},\ell_2),L_q(\mathcal{M},\ell_2))$ with a universal constant.
\end{theo}

Although this result was not known in the context of mixed spaces, in the context of row/column spaces, it is contained in \cite{Narcisse2023}[Proposition 3.19]. 

Now we turn to its proof. We keep the notation from Lemma \ref{AdaptedLemma0}. 
For every integer $n\pg1$, let $\mathcal{N}_n$ denote the von Neumann subalgebra of $\mathcal{N}$ generated by $\{\xi_k\ \mid\ 1\pp k\pp n\}$. Then, Assertion 0) of Lemma \ref{AdaptedLemma0} ensures that $(\mathcal{N}_n)_{n\pg1}$ is a filtration on $\mathcal{N}$. Moreover, if $(\FF_n)_{n\pg1}$ are the associated conditional expectations, for every $n,k\pg1$ we have
\[\FF_n(\xi_k)=\left\{\begin{array}{cl}
    \xi_k & {\rm if}\ k\pp n\\
    0 & {\rm otherwise} 
\end{array}\right..\]
Following the approach of Xu in \cite{XuBook} to establish the noncommutative version of Stein's inequality, for every integer $n\pg1$, we set
\[\mathcal{T}_{2n-1}:=\mathcal{M}_n\bar{\otimes}\mathcal{N}_n,\ \ \ \ \ \ \ \ \ \ \ \ \ \mathcal{T}_{2n}:=\mathcal{M}_{n+1}\bar{\otimes}\mathcal{N}_n.\]
Then, $(\mathcal{T}_n)_{n\pg1}$ is a filtration on $\mathcal{T}$. Let $(\DD_n)_{n\pg1}$ denote the associated difference projections. 

For $1\pp p\pp\infty$, we consider the following norm-closed subspace of $L_p(\mathcal{T})$,
\[\PP_p(\mathcal{T}):=\big\{y\in L_p(\mathcal{T})\ \mid\ \forall n\pg1,\ \DD_{2n}(y)=0\big\}.\]
Then, by applying the results from Example \ref{Example1}, namely Proposition \ref{Example 1 - Proposition 1} and Theorem \ref{Example 1 - Theorem 1}, to the current setting, with $I:=\{2n-1\ \mid\ n\pg1\}$, we know that if $1\pp p,q\pp\infty$, $(\PP_p(\mathcal{T}),\PP_q(\mathcal{T}))$ is $K$-complemented in $(L_p(\mathcal{T}),L_q(\mathcal{T}))$ with a universal constant. The following lemma is of fundamental importance.  

\begin{lemm}\label{AdaptedLemma1}
Let $x\in(L_1+L_\infty)(\mathcal{M},\ell_2)$. Then, $x$ is adapted if and only if $\DD_{2n}(Tx)=0$ for every $n\pg1$. In particular, if $1\pp p\pp\infty$, we have $Tx\in\PP_p(\mathcal{T})$ if and only if $x\in L_p^{{\rm ad}}(\mathcal{M},\ell_2)$.
\end{lemm}
\begin{proof}
Assertion 1) of Lemma \ref{AdaptedLemma0} ensures that the serie $\sum_{k\pg1}x_k\otimes\xi_k$ converges weakly to $Tx$ in $(L_1+L_\infty)(\mathcal{T})$. Thus, if $n\pg1$, we have
\begin{align*}
\DD_{2n}(Tx)&=\sum_{k\pg1}\DD_{2n}(x_k\otimes\xi_k)\\
&=\sum_{k\pg1}(\EE_{n+1}\otimes\FF_n-\EE_n\otimes\FF_n)(x_k\otimes\xi_k)\\
&=\sum_{k=1}^{n}(\EE_{n+1}-\EE_n)(x_k)\otimes\xi_k.\\
\end{align*}
We deduce that $\DD_{2n}(Tx)=0$ for every $n\pg1$ if and only if $(\EE_{n+1}-\EE_n)(x_k)=0$ for every $1\pp k\pp n$. The conclusion follows since for every $k\pg1$ we have
\[x_k=\EE_k(x_k)+\sum_{n\pg k}(\EE_{n+1}-\EE_n)(x_k)\]
where the sum converges weakly in $(L_1+L_\infty)(\mathcal{M})$.
\end{proof}

\begin{rem}
The above computation is done \cite{XuBook} in his proof of the noncommutative version of Stein's inequality. It retrospectively justifies the introduction of the filtration $(\mathcal{T}_n)_{n\pg1}$. This type of construction has already appeared in the context of classical martingale inequalities in the work of Stein (see \cite{Stein}).
\end{rem}

\begin{lemm}\label{AdaptedLemma2}
The operator $P:(L_1+L_\infty)(\mathcal{T})\to(L_1+L_\infty)(\mathcal{T})$ commutes with the conditional expectations associated with the filtration $(\mathcal{T}_n)_{n\pg1}$ and, in particular, stabilises $\PP_p(\mathcal{T})$ for every $1\pp p\pp\infty$.
\end{lemm}
\begin{proof}
As $\mathcal{A}$ is weak*-dense in $\mathcal{N}$, it suffices to show that $P$ commutes on the algebraic tensor product $(L_1+L_\infty)(\mathcal{M})\otimes\mathcal{A}$ with the conditional expectations associated to the filtration $(\mathcal{T}_n)_{n\pg1}$, namely $\EE_n\otimes\FF_n$ and $\EE_{n+1}\otimes\FF_n$, for $n\pg1$. But on that algebraic tensor product, using assertion 2) of Lemma \ref{AdaptedLemma0}, we can write $P=I\otimes Q$, where $I$ is the identity operator on $(L_1+L_\infty)(\mathcal{M})$ and $Q$ is the linear idempotent operator on $\mathcal{A}$ such that 
\[Q(\xi)=\left\{\begin{array}{cl}
    \xi & {\rm if}\ \xi\in\{\xi_n\ \mid\ n\pg1\}\\
    0 & {\rm otherwise} 
\end{array}\right.\]
for every $\xi\in X$. Thus, it is sufficient to show that $Q$ commutes with $\FF_n$ for $n\pg1$. This is straightforward.
\end{proof}

Now, we are able to complete the proof of the main result.

\begin{proof}[Proof of Theorem \ref{AdaptedTheorem}]
We fix $1\pp p,q\pp\infty$. Let $x\in L_p^{\rm ad}(\mathcal{M},\ell_2)+L_q^{{\rm ad}}(\mathcal{M},\ell_2)$ and $y\in L_p(\mathcal{M},\ell_2)$, $z\in L_q(\mathcal{M},\ell_2)$ be such that
\[x=y+z.\]
Then
\[Tx=Ty+Tz.\]
But from Lemma \ref{AdaptedLemma1} we know that $Tx\in\PP_p(\mathcal{T})+\PP_q(\mathcal{T})$. As $(\PP_p(\mathcal{T}),\PP_q(\mathcal{T}))$ is $K$-complemented in $(L_p(\mathcal{T}),L_q(\mathcal{T}))$ with a universal constant, we deduce that we can write
\begin{equation}\label{AdaptedEquation1}
Tx=y'+z',
\end{equation}
where $y'\in \PP_p(\mathcal{T})$, $z'\in\PP_q(\mathcal{N})$, with 
\begin{equation}\label{AdaptedEquation2}
\|y'\|_{L_p(\mathcal{T})}\pp C\|Ty\|_{L_p(\mathcal{T})},\ \ \ \ \ \ \ \ \ \ \|z'\|_{L_q(\mathcal{T})}\pp C\|Tz\|_{L_q(\mathcal{T})},
\end{equation}
where $C>0$ is a universal constant. As $Tx$ is in the range of $P$, applying $P$ to both sides of \eqref{AdaptedEquation1} yields
\[Tx=Py'+Pz'.\]
Finally, we set
\[y'':=T^{-1}Py'\in L_p(\mathcal{M},\ell_2),\ \ \ \ \ \ \ \ \ \ \ \ \ \ \ \ \ z'':=T^{-1}Pz'\in L_q(\mathcal{M},\ell_2).\]
(recall that $T^{-1}$ is well defined on the range of $P$, which is also the range of $T$). Then we clearly have
\[x=y''+z''.\]
Furthermore, from Lemma \ref{AdaptedLemma2}, we know that $Ty''=Py'\in\PP_p(\mathcal{T})$ and $Tz''=Pz'\in\PP_q(\mathcal{T})$. Thus, from Lemma \ref{AdaptedLemma1}, we deduce that $y''\in L_p^{\rm ad}(\mathcal{M},\ell_2)$ and $z''\in L_q^{\rm ad}(\mathcal{M},\ell_2)$. Moreover, using the bounded properties of the two operators $T$ and $P$, combined with \eqref{AdaptedEquation2}, we obtain the following estimates,
\[\|y''\|_{L_p(\mathcal{M},\ell_2)}\pp\|Ty''\|_{L_p(\mathcal{T})}=\|Py'\|_{L_p(\mathcal{T})}\pp2\|y'\|_{L_p(\mathcal{T})}\pp 2C\|Ty\|_{L_p(\mathcal{T})}\pp4C\|y\|_{L_p(\mathcal{M},\ell_2)},\]
and
\[\|z''\|_{L_q(\mathcal{M},\ell_2)}\pp\|Tz''\|_{L_q(\mathcal{T})}=\|Pz'\|_{L_q(\mathcal{T})}\pp2\|z'\|_{L_q(\mathcal{T})}\pp 2C\|Tz\|_{L_q(\mathcal{T})}\pp4C\|z\|_{L_q(\mathcal{M},\ell_2)}.\]
The proof is then complete.
\end{proof}

\begin{rem}
By using the fact that $(\PP_p(\mathcal{T}),\PP_q(\mathcal{T}))$ is complemented in $(L_p(\mathcal{T}),L_q(\mathcal{T}))$ whenever $1<p,q<\infty$, and by following the same approach one essentially recovers the proof by Xu in \cite{XuBook} of the noncommutative version of Stein's inequality, originally proved in \cite{PisierXuMartingales}, which asserts that if $1<p<\infty$ and $x=(x_n)_{n\pg1}$ belongs to $L_p(\mathcal{M},\ell_2)$, then the sequence $y=(\EE_n(x_n))_{n\pg1}$ also belongs to $L_p(\mathcal{M},\ell_2)$, and 
\begin{equation}\label{Stein}
\|y\|_{L_p(\mathcal{M},\ell_2)}\pp C_p\|x\|_{L_p(\mathcal{M},\ell_2)}
\end{equation}
where $C_p>0$ depends only on $p$. 
\end{rem} 

Now, if $E$ is an fully symmetric intermediate space for $(L_1,L_\infty)$, we consider the following subspace of $E(\mathcal{M},\ell_2)$,
\[E^{{\rm ad}}(\mathcal{M},\ell_2):=\big\{x\in E(\mathcal{M},\ell_2)\ \mid\ x{\rm\ is\ an\ adapted\ sequence}\big\}.\]
Notice that $E^{{\rm ad}}(\mathcal{M},\ell_2)$ is norm-closed in $E(\mathcal{M},\ell_2)$, as a consequence of the fact $E(\mathcal{M},\ell_2)$ is an interpolation space for $(L_1(\mathcal{M},\ell_2),L_\infty(\mathcal{M},\ell_2))$. Then, using Theorem \ref{AdaptedTheorem}, we easily deduce the following result.

\begin{theo}
Let $\Phi$ be a $K$-parameter space such that the exact interpolation space $E:=K_\Phi(L_1,L_\infty)$ satisfies the separability condition \hyperlink{separability}{(S)}. Then
\[K_\Phi(L_1^{\rm ad}(\mathcal{M},\ell_2),L^{\rm ad}_\infty(\mathcal{M},\ell_2))=E^{\rm ad}(\mathcal{M},\ell_2)\]
with equivalent norms and constants independent of $\mathcal{M}$.  
\end{theo}
\begin{proof}
Using Theorem \ref{FunctorRC}, we clearly have a continuous inclusion
\[K_\Phi(L_1^{\rm ad}(\mathcal{M},\ell_2),L^{\rm ad}_\infty(\mathcal{M},\ell_2))\subset E^{\rm ad}(\mathcal{M},\ell_2)\]
which has a closed range by $K$-closedness. Thus, it suffices to show that it also has dense range. Let $x=(x_n)_{n\pg1}$ be a eventually-zero sequence of $E(\mathcal{M})$ which is adapted. Then, clearly $x\in(L_1^{\rm ad}(\mathcal{M},\ell_2)+L^{\rm ad}_\infty(\mathcal{M},\ell_2))\cap E(\mathcal{M})$. However, again using Theorem \ref{FunctorRC} and Proposition \ref{Kclosedness}, we know that
\[(L_1^{\rm ad}(\mathcal{M},\ell_2)+L^{\rm ad}_\infty(\mathcal{M},\ell_2))\cap E(\mathcal{M})=K_\Phi(L_1^{\rm ad}(\mathcal{M},\ell_2),L^{\rm ad}_\infty(\mathcal{M},\ell_2)).\]
Thus, $x\in K_\Phi(L_1^{\rm ad}(\mathcal{M},\ell_2),L^{\rm ad}_\infty(\mathcal{M},\ell_2))$. This concludes the proof because the space of eventually-zero sequences of $E(\mathcal{M})$ that are adapted is clearly norm-dense in $E^{{\rm ad}}(\mathcal{M},\ell_2)$ by the separability condition.
\end{proof}

In particular, we obtain the following corollary. In the setting of column/row spaces, another proof of this result is contained in \cite{Narcisse2021}[Theorem 3.13].

\begin{coro}
For every $0<\theta<1$, we have
\[(L_1^{\rm ad}(\mathcal{M},\ell_2),L_\infty^{\rm ad}(\mathcal{M},\ell_2))_{\theta,p}=L_p^{\rm ad}(\mathcal{M},\ell_2)\]
with equivalent norms and constants independent of $\mathcal{M}$, where $1/p=1-\theta$.
\end{coro}

\subsection{Main results: Noncommutative Hardy spaces}

Let $\mathcal{M}$ be a von Neumann algebra endowed with a n.s.f. trace and a filtration $(\mathcal{M}_n)_{n\pg1}$ with associated conditional expectations $(\EE_n)_{n\pg1}$. Recall that a sequence $x=(x_n)_{n\pg1}$ of $(L_1+L_\infty)(\mathcal{M})$ is said to be a martingale difference if it is adapted and $\EE_{n-1}(x_n)=0$ for every $n\pg2$.

For every $1\pp p\pp\infty$, we consider the \textit{column/row Hardy spaces},
\[H_p^{c}(\mathcal{M}):=\big\{x\in L_p(\mathcal{M},\ell_2^c)\ \mid\ x{\rm\ is\ a\ martingale\ difference\ sequence}\big\},\]
\[H_p^{r}(\mathcal{M}):=\big\{x\in L_p(\mathcal{M},\ell_2^r)\ \mid\ x{\rm\ is\ a\ martingale\ difference\ sequence}\big\},\]
and the \textit{mixed Hardy space},
\[H_p^{rc}(\mathcal{M}):=\big\{x\in L_p(\mathcal{M},\ell_2^{rc})\ \mid\ x{\rm\ is\ a\ martingale\ difference\ sequence}\big\}.\]
It is clear that these row/column/mixed Hardy spaces are closed in the corresponding row/column/mixed spaces. 

\begin{rem}
For mixed spaces, this is not the usual definition. If $1\pp p\pp\infty$, the \textit{usual mixed Hardy space} denoted $\mathcal{H}_p(\mathcal{M})$ is defined as follows,
\[\mathcal{H}_p(\mathcal{M}):=\left\{\begin{array}{cl}
    H_p^r(\mathcal{M})\cap H_p^c(\mathcal{M}) & {\rm if}\ 2\pp p\pp\infty \\
    H_p^r(\mathcal{M})+H_p^c(\mathcal{M}) & {\rm if}\ 1\pp p<2 
\end{array}\right.\]
and is equipped with the corresponding intersection/sum norm. When $2\pp p\pp\infty$, it is clear from the definitions that we have $\mathcal{H}_p(\mathcal{M})=H_p^{rc}(\mathcal{M})$ with the same norms. When $1<p<2$, we have $\mathcal{H}_p(\mathcal{M})=H_p^{rc}(\mathcal{M})$ with equivalent norms, as a consequence of the noncommutative version of Stein's inequality \eqref{Stein}. However, when $p=1$, we only have a contractive inclusion $\mathcal{H}_1(\mathcal{M})\subset H_1^{rc}(\mathcal{M})$.
\end{rem}

The following result is completely new, even in the setting of column/row spaces. 

\begin{theo}\label{MartingaledifferenceTheorem}
Let $1\pp p,q\pp\infty$. Then $(H_p(\mathcal{M}),H_q(\mathcal{M}))$ is $K$-complemented in $(L_p(\mathcal{M},\ell_2),L_q(\mathcal{M},\ell_2))$ with a universal constant.
\end{theo}

Again, we will make use of the notation from Lemma \ref{AdaptedLemma0}. We recall that if $n\pg1$ then $\mathcal{N}_n$ denotes the von Neumann subalgebra of $\mathcal{N}$ generated by $\{\xi_k\ \mid\ 1\pp k\pp n\}$. Moreover, $(\mathcal{N}_n)_{n\pg1}$ is a filtration on $\mathcal{N}$, and if $(\FF_n)_{n\pg1}$ are the associated conditional expectations, for every $n,k\pg1$ we have
\[\FF_n(\xi_k)=\left\{\begin{array}{cl}
    \xi_k & {\rm if}\ k\pp n\\
    0 & {\rm otherwise} 
\end{array}\right..\]
For every $n\pg1$, we set
\[\mathcal{T}_{2n-1}^{(0)}:=\mathcal{M}_n\bar{\otimes}\mathcal{N}_n,\ \ \ \ \ \ \ \ \ \ \ \ \ \mathcal{T}_{2n}^{(0)}:=\mathcal{M}_{n+1}\bar{\otimes}\mathcal{N}_n,\]
and 
\[\mathcal{T}_{2n-1}^{(1)}:=\mathcal{M}_n\bar{\otimes}\mathcal{N}_{n+1},\ \ \ \ \ \ \ \ \ \ \ \ \ \mathcal{T}_{2n}^{(1)}:=\mathcal{M}_{n+1}\bar{\otimes}\mathcal{N}_{n+1}.\]
Then $(\mathcal{T}_n^{(0)})_{n\pg1}$ and $(\mathcal{T}_n^{(1)})_{n\pg1}$ are two filtrations on $\mathcal{T}$. Let $(\EE_n^{(0)})_{n\pg1},(\EE_n^{(1)})_{n\pg1}$ denote their associated conditional expectations and let $(\DD_n^{(0)})_{n\pg1}$, $(\DD_n^{(1)})_{n\pg1}$ denote their associated difference projections. Thus, we have
\[\EE_{2n-1}^{(0)}:=\EE_n\otimes\FF_n,\ \ \ \ \ \ \ \ \ \ \ \ \ \EE_{2n}^{(0)}:=\EE_{n+1}\otimes\FF_n,\]
and 
\[\EE_{2n-1}^{(1)}:=\EE_n\otimes\FF_{n+1},\ \ \ \ \ \ \ \ \ \ \ \ \ \EE_{2n}^{(1)}:=\EE_{n+1}\otimes\FF_{n+1}.\]
for every $n\pg1$. For $1\pp p\pp\infty$, we consider the following norm-closed subspace of $L_p(\mathcal{T})$,
\[\QQ_p(\mathcal{T}):=\big\{y\in L_p(\mathcal{T})\ \mid\ \forall n\pg1,\ \DD_{2n}^{(0)}(y)=\DD_{2n-1}^{(1)}(y)=0\big\}.\]
Then, by applying Theorem \ref{Example 3 - Theorem 1}, we know that if $1\pp p,q\pp\infty$, then $(\QQ_p(\mathcal{T}),\QQ_q(\mathcal{T}))$ is $K$-complemented in $(L_p(\mathcal{T}),L_q(\mathcal{T}))$ with a universal constant. A martingale difference sequence $x=(x_n)_{n\pg1}$ is \textit{centered} if $x_1=0$.

\begin{lemm}\label{CenteredMartingaledifferenceLemma1}
Let $x\in L_1(\mathcal{M},\ell_2)+L_\infty(\mathcal{M},\ell_2)$. Then, $x$ is a centered martingale difference sequence if and only if $\DD_{2n}^{(0)}(y)=\DD_{2n-1}^{(1)}(y)=0$ for every $n\pg1$. In particular, if $1\pp p\pp\infty$, we have $Tx\in\QQ_p(\mathcal{T})$ if and only if $x\in L_p(\mathcal{M},\ell_2)$ and $x$ is a centered martingale difference sequence.
\end{lemm}
\begin{proof}
We have 
\begin{align*}
\DD_{1}^{(1)}(Tx)&=\sum_{k\pg1}\DD_{1}^{(1)}(x_k\otimes\xi_k)=\sum_{k\pg1}\EE_{1}^{(1)}(x_k\otimes\xi_k)\\
&=\sum_{k\pg1}(\EE_{1}\otimes\FF_2)(x_k\otimes\xi_k)=\EE_1(x_1)\otimes\xi_1+\EE_1(x_2)\otimes\xi_2,\\
\end{align*}
and if $n\pg2$ we have
\begin{align*}
\DD_{2n-1}^{(1)}(Tx)&=\sum_{k\pg1}\DD_{2n-1}^{(1)}(x_k\otimes\xi_k)\\
&=\sum_{k\pg1}(\EE_{2n-1}^{(1)}-\EE_{2(n-1)}^{(1)})(x_k\otimes\xi_k)\\
&=\sum_{k\pg1}(\EE_{n}\otimes\FF_{n+1}-\EE_n\otimes\FF_n)(x_k\otimes\xi_k)\\
&=\sum_{k\pg1}\EE_n(x_k)\otimes(\FF_{n+1}-\FF_n)(\xi_k)\\
&=\EE_n(x_{n+1})\otimes\xi_{n+1}.
\end{align*}
We deduce that $\DD_{2n-1}^{(1)}(Tx)=0$ for every $n\pg1$ if and only if $\EE_1(x_1)=0$ and $\EE_n(x_{n+1})=0$ for every $n\pg1$. In addition, in Lemma \ref{AdaptedLemma1} we have already shown that $x$ is adapted if and only if $\DD_{2n}^{(0)}(y)=0$ for every $n\pg1$. The conclusion follows.
\end{proof}

\begin{lemm}\label{CenteredMartingaledifferenceLemma2}
The operator $P:(L_1+L_\infty)(\mathcal{T})\to(L_1+L_\infty)(\mathcal{T})$ commutes with the conditional expectations associated with the two filtrations $(\mathcal{T}_n^{(0)})_{n\pg1}$ and $(\mathcal{T}_n^{(1)})_{n\pg1}$ and, in particular, stabilises $\QQ_p(\mathcal{T})$ for every $1\pp p\pp\infty$.
\end{lemm}
\begin{proof}
The arguments are the same as in the proof of Lemma \ref{AdaptedLemma2}.
\end{proof}

Now we can complete the proof of the main result.

\begin{proof}[Proof of Theorem \ref{MartingaledifferenceTheorem}]
If $x=(x_n)_{n\pg1}$ is a sequence of $(L_1+L_\infty)(\mathcal{M})$, we will denote $Ax$ (resp. $Bx$) the sequence $(y_n)_{n\pg1}$ such that $y_1=0$ and $y_n=x_n$ for every $n\pg2$ (resp. $y_1=\EE_1(x_1)$ and $y_n=0$ for every $n\pg2$). Notice that if $1\pp p\pp\infty$ and $x\in L_p(\mathcal{M},\ell_2)$, then $Ax\in L_p(\mathcal{M},\ell_2)$ and $Bx\in L_p(\mathcal{M},\ell_2)$, with
\[\|Ax\|_{L_p(\mathcal{M},\ell_2)}\pp\|x\|_{L_p(\mathcal{M},\ell_2)},\ \ \ \ \ \ \ \ \ \ \\\|Bx\|_{L_p(\mathcal{M},\ell_2)}\pp\|x\|_{L_p(\mathcal{M},\ell_2)}.\]
Now, fix $1\pp p,q\pp\infty$. Let $G_p(\mathcal{M})$, $G_q(\mathcal{M})$ denote the norm-closed subspace of respectively $H_p(\mathcal{M})$, $H_q(\mathcal{M})$ formed by centered martingale difference sequences, which also coincides with the image under the operator $A$ of respectively $H_p(\mathcal{M})$, $H_q(\mathcal{M})$. Then, knowing that $(\QQ_p(\mathcal{T}),\QQ_q(\mathcal{T}))$ is $K$-complemented in $(L_p(\mathcal{T}),L_q(\mathcal{T}))$ with a universal constant, and using Lemma \ref{CenteredMartingaledifferenceLemma1} and Lemma \ref{CenteredMartingaledifferenceLemma2}, as in the proof of Theorem \ref{AdaptedTheorem}, one shows that $(G_p(\mathcal{M}),G_q(\mathcal{M}))$ is $K$-complemented in $(L_p(\mathcal{M},\ell_2),L_q(\mathcal{M},\ell_2))$ with a universal constant. From this we will deduce that $(H_p(\mathcal{M}),H_q(\mathcal{M}))$ is $K$-complemented in $(L_p(\mathcal{M},\ell_2),L_q(\mathcal{M},\ell_2))$. Let $x\in H_p(\mathcal{M})+H_q(\mathcal{M})$ and $y\in L_p(\mathcal{M},\ell_2)$, $z\in L_q(\mathcal{M},\ell_2)$ be such that
\[x=y+z.\]
Then
\[Ax=Ay+Az.\]
As $Ax\in G_p(\mathcal{M})+G_q(\mathcal{M})$, by $K$-complementation we deduce that we can write
\[Ax=y'+z',\]
where $y'\in G_p(\mathcal{M})$, $z'\in G_q(\mathcal{M})$, with
\[\|y'\|_{L_p(\mathcal{M},\ell_2)}\pp C\|Ay\|_{L_p(\mathcal{M},\ell_2)},\ \ \ \ \ \ \|z'\|_{L_p(\mathcal{M},\ell_2)}\pp C\|Az\|_{L_p(\mathcal{M},\ell_2)},\]
where $C>0$ is a universal constant. Finally, we set
\[y'':=y'+By\in L_p(\mathcal{M},\ell_2),\ \ \ \ \ \ \ \ \ z'':=z'+Bz\in L_q(\mathcal{M},\ell_2).\]
Then we have clearly
\[x=y''+z''.\]
Moreover, $y''$ and $z''$ are of course martingale difference sequences, and by using the triangle inequality and the previous estimates, we get
\[\|y''\|_{L_p(\mathcal{M},\ell_2)}\pp\|y'\|_{L_p(\mathcal{M},\ell_2)}+\|By\|_{L_p(\mathcal{M},\ell_2)}\pp(C+1)\|y\|_{L_p(\mathcal{M},\ell_2)}\]
and
\[\|z''\|_{L_q(\mathcal{M},\ell_2)}\pp\|z'\|_{L_q(\mathcal{M},\ell_2)}+\|Bz\|_{L_q(\mathcal{M},\ell_2)}\pp(C+1)\|z\|_{L_q(\mathcal{M},\ell_2)}.\]
The proof is complete.
\end{proof}

Now, if $E$ is a fully symmetric intermediate space for $(L_1,L_\infty)$, we consider the following norm-closed subspace of $E(\mathcal{M},\ell_2)$,
\[H_E(\mathcal{M}):=\big\{x\in E(\mathcal{M},\ell_2)\ \mid\ x{\rm\ is\ a\ martingale\ difference\ sequence}\big\}.\]
Then, as in the previous paragraph, using Theorem \ref{MartingaledifferenceTheorem} we easily deduce the following result.

\begin{theo}
Let $\Phi$ be a $K$-parameter space such that the exact interpolation space $E:=K_\Phi(L_1,L_\infty)$ satisfies the separability condition \hyperlink{separability}{(S)}. Then
\[K_\Phi(H_1(\mathcal{M}),H_\infty(\mathcal{M}))=H_E(\mathcal{M})\]
with equivalent norms, and constants independant of $\mathcal{M}$.
\end{theo}

In particular, we obtain the following corollary. In the row/column setting, it has been obtained in \cite{Narcisse2024}[Theorem 3.5].

\begin{coro}
For every $0<\theta<1$, we have
\[(H_1(\mathcal{M}),H_\infty(\mathcal{M}))_{\theta,p}=H_p(\mathcal{M})\]
with equivalent norms and constants independent of $\mathcal{M}$, where $1/p=1-\theta$.
\end{coro}

\section*{Concluding remarks}

In this paper, we did not work within the framework of noncommutative Lebesgue spaces with exponents in the range $(0,1)$, which are only quasi-Banach spaces. The method we use in our article is not suited to deal with quasi-Banach spaces, as it relies heavily on duality results that are only valid in the Banach space range. 

In this paper, we did not focus on the class of noncommutative conditioned (Hardy) spaces as defined by Junge and
Xu in \cite{JungeXuBurkholder} in connection with the noncommutative version of the Burkholder-Rosenthal inequalities. However, let us mention briefly that the conditioned version of Theorem \ref{MartingaledifferenceTheorem} in the column or row spaces setting, which is contained in \cite{Narcisse2023}[Theorem 3.1] can also be obtained by following the same reasoning as in the proof of Theorem \ref{AdaptedTheorem}, by replacing the adapted sequences with the matrices whose columns are adapted. 

\section*{Acknowledgements}

I am very grateful to my advisor Éric Ricard for many valuable discussions and his guidance throughout the writing of this article.

The author was supported by ANR-19-CE40-0002.

\bibliographystyle{plain}
\bibliography{bibliography} 

\end{document}